\documentclass[12pt, a4paper]{amsart}
\usepackage{amssymb}
\usepackage{tikz} 
\usepackage{mathabx}
\usepackage{graphicx}
\usepackage{mathrsfs}
\usepackage{hyperref}
\usepackage{import}
\usepackage{tcolorbox}

\theoremstyle{plain}

\newtheorem{theorem}{Theorem}[section]
\newtheorem{lemma}[theorem]{Lemma}

\newtheorem{corollary}[theorem]{Corollary}

\theoremstyle{definition}
\newtheorem{define}{Definition}[section]

\newtheorem{example}{Example}[section]

\theoremstyle{remark}
\newtheorem{remark}{Remark}[section]

\newtheorem*{acknowledgement}{Acknowledgement}

\begin{document}

\date\today

\title[Asymptotic behavior of the Bergman kernel]{Asymptotic behavior of the Bergman kernel and associated invariants in weakly pseudoconvex domains}
\author{Ninh Van Thu}

\address{Ninh Van Thu}
\address{Faculty of Mathematics and Informatics, Hanoi University of Science and Technology, No. 1 Dai Co Viet, Bach Mai, Hanoi, Vietnam}
\email{thu.ninhvan@hust.edu.vn}

\subjclass[2020]{Primary 32H02; Secondary 32T15, 32M05.}
\keywords{Automorphism group, scaling method, $h$-extendible domain}
\begin{abstract} In this paper, we present an explicit description for the boundary behavior of the Bergman kernel function, the Bergman metric, and the associated curvatures along certain sequences converging to an $h$-extendible boundary point.
\end{abstract}
\maketitle

\section{Introduction}
Let $\Omega$ be a domain in $\mathbb C^n$ and let $\mathrm{Aut}(\Omega)$ denote the set of all automorphisms of $\Omega$. For strongly pseudoconvex domains in $\mathbb{C}^n$, C. Fefferman \cite{Fe74} established the asymptotic expansion formula of the Bergman kernel function, which provides a complete asymptotic expansion of the Bergman kernel near strongly pseudoconvex boundary points, revealing the precise relationship between the boundary geometry and the analytic structure. Subsequently,  based on this formula, Klembeck \cite{Kl78} showed that the holomorphic sectional curvature of a $\mathcal{C}^\infty$-smooth strongly pseudoconvex bounded domain in $\mathbb{C}^n$ approaches $-4/(n+1)$, that of the unit ball, near the boundary. This result was optimally generalized by \cite{KY96} for $\mathcal{C}^2$-smooth strongly pseudoconvex bounded domains in $\mathbb{C}^n$. For more comprehensive results on curvatures of the Bergman metric, we refer the reader to \cite{AS83, BSY95, GK81, He07, HLT25, KS02, KK03, Le71, Mc89a, SX25, Wa17, Zh17} and the references therein.

Many results have been obtained for estimates of the Bergman kernel on the diagonal and the Bergman metric along sequences converging nontangentially to the boundary. We first recall that for $(n + 1)$-dimensional domains of the form
\begin{equation*}
\Omega_F = \{p=(z, w) \in \mathbb{C}^n \times \mathbb{C} : \mathrm{Im}(w) > F(z)\},
\end{equation*}
where $F \colon \mathbb{C}^n \to \mathbb{R}$ is $\mathcal{C}^\infty$-smooth and plurisubharmonic satisfying that $F(0) = 0 = \nabla F(0)$, J.~Kamimoto \cite{Ka04, Ka24} showed that
\begin{equation}\label{est-bergman-1}
K_{\Omega_F}(p, p ) \approx \frac{1}{d_{\Omega_F}(p)^{2+2/d_F}(\log(1/d_{\Omega_F}(p)))^{m_F-1}}
\end{equation}
on transversal approach paths to $\xi_0=(0',0)) \in \partial \Omega_F$, where $d_F$ and $m_F$ denote the Newton distance and multiplicity, respectively (see \cite{Ka04, Ka24} for these definitions). Here and in what follows, $d_\Omega(z)$ denotes the Euclidean distance from $z$ to the boundary $\partial \Omega$. In addition, $\lesssim$ and $\gtrsim$ denote inequality up to a positive constant and we use $\approx$ for the combination of $\lesssim$ and $\gtrsim$. This result generalizes the classical estimates previously obtained for specific boundary types: $d_F = n/2$, $m_F = 1$ if $\xi_0$ is strongly pseudoconvex (cf.\ \cite{BS76, Fe74, Gra75, Ho65}), and $d_F = \sum_{k=1}^n \frac{1}{2m_k}$, $m_F = 1$ if $\xi_0$ is $h$-extendible with multitype $\mathcal{M}(\xi_0) = (2m_1, \ldots, 2m_n, 1)$ (cf.\ \cite[Theorem $1$]{BSY95} and \cite{Cat89, HS22} for two-dimensional weakly pseudoconvex domains).

Next, in the case when $\Omega\subset \mathbb C^{n+1}$ is $h$-extendible at $\xi_0\in \partial \Omega$ with multitype $\mathcal{M}(\xi_0) = (2m_1, \ldots, 2m_n, 1)$, H.P.~Boas et al.\ \cite[Theorem $2$]{BSY95} proved that
\begin{equation}\label{BMmetric-1}
d^2_{\Omega}(z;\xi) \approx \sum_{k=1}^n \frac{|\xi_k|^2}{d_\Omega(z)^{1/(2m_k)}} + \frac{|\xi_{n+1}|^2}{d_\Omega(z)^2}
\end{equation}
on transversal approach paths to $\xi_0$, where $\displaystyle\xi=\sum_{k=1}^n \xi_k\dfrac{\partial}{\partial z_k}+\xi_{n+1}\dfrac{\partial}{\partial w}\in T^{1,0}_{(z,w)} \Omega\setminus \{0\}$ (cf. \cite{Gra75} for strongly pseudoconvex domains).

The first aim of this paper is to prove the following theorem, which enables us to describe explicitly the boundary behavior of the Bergman kernel on the diagonal, the Bergman metric, and the associated curvatures along a sequence converging uniformly $\Lambda$-tangentially to a strongly $h$-extendible boundary point (cf. Definition \ref{strongly-h-extendible} and Definition \ref{lambda-tangent} in Section \ref{S-h-extendible}, respectively).
 
\begin{theorem}\label{holo-sec-curvature}
Let $\Omega$ be a bounded domain in $\mathbb{C}^{n+1}$ with $C^\infty$-smooth boundary and $\xi_0\in\partial \Omega$ be strongly $h$-extendible with Catlin's finite multitype $(2m_1,\ldots,2m_n, 1)$ (cf. Definition \ref{strongly-h-extendible}).  Denote by $\Lambda=(1/2m_1, \ldots, 1/2m_n)$. If $\{\eta_j=(\alpha_j,\beta_j)\}\subset \Omega $ is a sequence converging uniformly $\Lambda$-tangentially to  $\xi_0 \in \partial \Omega$ (Definition \ref{lambda-tangent}), then we have
\begin{align*}
K_{\Omega}(\eta_j, \eta_j) &\approx \frac{1}{ (\tau_{j1} \cdots \tau_{jn})^2 \epsilon_j^2}\text{ with } \epsilon_j\approx d_\Omega(\eta_j),\tau_{jk}:=|\alpha_{jk}|.\Big(\dfrac{\epsilon_j}{|\alpha_{jk}|^{2m_k}}\Big)^{1/2}, \;1\leq k\leq n;\\
d^2_\Omega(\eta_j;\xi)&\approx \frac{|\xi_{n+1}|^2}{\epsilon_j^2}+\sum_{k=1}^n \max\{\ell_{jk},1\}\;\frac{|\xi_k|^2}{\tau_{jk}^2}\text{ with } \ell_{jk}\approx \left(\epsilon_j^{-1}\tau_{jk}\left|\frac{\partial\rho(\eta_j)}{\partial z_k}\right|\right)^2, \;1\leq k\leq n; \\
\lim_{j\to\infty} \mathrm{Sec}_{\Omega}&(\eta_j;\xi) = -\frac{4}{n+2};\lim_{j\to\infty} \mathrm{Ric}_{\Omega}(\eta_j;\xi) =-1; \lim_{j\to\infty} \mathrm{Scal}_{\Omega}(\eta_j)=-(n+1),
\end{align*}
where $K_{\Omega}(p, p),  d^2_\Omega(p;\xi),\mathrm{Sec}_{\Omega}(p,\xi),\mathrm{Ric}_{\Omega}(p,\xi)$, and $\mathrm{Scal}_{\Omega}(p) $ respectively denote the Bergman kernel, the Bergman metric, the holomorphic sectional curvature, the Ricci curvature, and the scalar curvature of $\Omega$ at $p=(z,w)\in \mathbb C^n\times \mathbb C$ in the direction $\displaystyle\xi=\sum_{k=1}^n \xi_k\dfrac{\partial}{\partial z_k}+\xi_{n+1}\dfrac{\partial}{\partial w}\in T^{1,0}_{p} \Omega\setminus \{0\}$.
\end{theorem}

In what follows, let us denote by $\rho$ the local defining function for $\Omega$ near $\xi_0$. Then, in the case when $\displaystyle \left|\alpha_{jk}\frac{\partial \rho(\eta_j)}{\partial z_k}\right| \approx |\alpha_{j1}|^{2m_1}$ for every $1\leq k\leq n$, i.e., when $\{\eta_j\}\subset \Omega $ satisfies the $(B,\xi_0)$-condition (cf. Definition \ref{eq:balanced-condition}), we obtain the following corollary.
\begin{corollary}\label{cor:higher-dim-bergman}
Under the same hypotheses as in Theorem~\ref{holo-sec-curvature}, assume also that  $\{\eta_j\}\subset \Omega $ satisfies the $(B,\xi_0)$-condition (cf. Definition \ref{eq:balanced-condition}). Then the Bergman metric admits the asymptotic expansion
\begin{equation}\label{eq:higher-dim-metric}
d^2_\Omega(\eta_j;\xi)\approx \frac{|\xi_{n+1}|^2}{\epsilon_j^2} + \ell_j \sum_{k=1}^n \frac{|\xi_k|^2}{\tau_{jk}^2},
\end{equation}
for all $\displaystyle\xi=\sum_{k=1}^n \xi_k\dfrac{\partial}{\partial z_k}+\xi_{n+1}\dfrac{\partial}{\partial w}\in T^{1,0}_{\eta_j} \Omega\setminus \{0\}$, where $\epsilon_j\approx d_\Omega(\eta_j)$ and   $\displaystyle \ell_j := \frac{|\alpha_{j1}|^{2m_1}}{\epsilon_j} \to +\infty$ as $j \to \infty$.
\end{corollary}

We emphasize that Theorem \ref{holo-sec-curvature} and Corollary \ref{cor:higher-dim-bergman} point out that the boundary behavior of the Bergman kernel on the diagonal and the Bergman metric along sequences converging tangentially to the boundary are quite different from (\ref{est-bergman-1}) and (\ref{BMmetric-1}) respectively, such as 
\begin{align*}
K_{\mathcal{E}_{1,2,3}}(\eta_j, \eta_j)&\approx \dfrac{1}{\left(d_{\mathcal{E}_{1,2,3}}(\eta_j)\right)^{2+3/4+2/3}}\not \approx \dfrac{1}{\left(d_{\mathcal{E}_{1,2,3}}(\eta_j)\right)^{2+1/4+1/6}};\\
d^2_{\mathcal{E}_{1,2,3}}(\eta_j;\xi)&\approx  \frac{|\xi_3|^2}{d_{\mathcal{E}_{1,2,3}}(\eta_j)^2}+\frac{|\xi_1|^2}{d_{\mathcal{E}_{1,2,3}}(\eta_j)^{5/4}}+\frac{|\xi_2|^2}{d_{\mathcal{E}_{1,2,3}}(\eta_j)^{7/6}}\\
&\not \approx  \frac{|\xi_3|^2}{d_{\mathcal{E}_{1,2,3}}(\eta_j)^2}+\frac{|\xi_1|^2}{d_{\mathcal{E}_{1,2,3}}(\eta_j)^{1/4}}+\frac{|\xi_2|^2}{d_{\mathcal{E}_{1,2,3}}(\eta_j)^{1/6}}
\end{align*}
for $\displaystyle\xi= \xi_1\dfrac{\partial}{\partial z_1}+\xi_2\dfrac{\partial}{\partial z_2}+\xi_{3}\dfrac{\partial}{\partial w}\in T^{1,0}_{\eta_j}\mathcal {E}_{1,2,3}\setminus \{0\}$, where 
$$
\mathcal {E}_{1,2,3}:=\left\{(z_1,z_2,w)\in \mathbb C^{3}\colon \mathrm{Re}(w)+ |z_1|^4+|z_2|^6 <0\right\}
$$ 
and $\eta_j=\big(1/j^{1/4}, 1/j^{1/6},-2/j-1/j^2\big)\in \mathcal{E}_{1,2,3},\; j\in \mathbb N_{\geq 1}$ (see Example \ref{ellipsoid-strongly-h-extendible} in Section \ref{S-h-extendible} for more details). 

Furthermore, S.G. Krantz and J. Yu \cite{KrY96} established the existence of nontangential limits of curvatures of the Bergman metric (see also \cite[Theorem $2$]{BSY95}). Moreover, the condition on nontangential convergences in these limits cannot be removed. In fact, the results given in \cite{AS83} demonstrate this phenomenon. However, Theorem \ref{holo-sec-curvature} yields that the curvatures of the Bergman metric approach those of the unit ball $\mathbb{B}^{n+1}$ along sequences converging uniformly $\Lambda$-tangentially to a strongly $h$-extendible boundary point.

Now we turn our attention to bounded pseudoconvex domains in $\mathbb{C}^2$. Let $\xi_0\in \partial \Omega$ be pseudoconvex of finite D'Angelo type. Then, following the proofs given in \cite{Be03} (or in \cite{BP89} for the real-analytic boundary case), one concludes that for each sequence $\{\eta_j\}\subset \Omega$ that converges to $\xi_0$, there exists a scaling sequence $\{F_j\}\subset\mathrm{Aut}(\mathbb{C}^2)$ such that $F_j(\eta_j)$ converges to $(0,-1)$ and, without loss of generality, $F_j(\Omega)$ converges normally to a model
$$
M_P=\{(z,w)\in \mathbb{C}^2\colon \mathrm{Re}(w)+P(z)<0\},
$$
where $P$ is a subharmonic polynomial of degree $\leq 2m$, with $2m$ being the D'Angelo type of $\partial\Omega$ at $\xi_0$, and $P$ has no harmonic terms. We note that the local model $M_P$ depends essentially on the boundary behavior of the sequence $\{\eta_j\}$.

The second part of this paper deals with the case where the sequence $\{\eta_j\}$ accumulates at $\xi_0$ very tangentially to $\partial \Omega$ (see Definition \ref{spherically-convergence}) so that $M_P$ is biholomorphically equivalent to the unit ball $\mathbb{B}^2$, i.e., $\deg P=2$. More precisely, the second aim of this paper is to prove the following theorem, which enables us to describe explicitly the boundary behavior of the Bergman kernel on the diagonal, the Bergman metric, and the associated curvatures along a sequence converging spherically $\frac{1}{2m}$-tangentially to a finite-type boundary point (cf. Definition \ref{spherically-convergence} in Section \ref{S4}).
\begin{theorem}\label{holo-sec-curvature-two}
Let $\Omega$ be a bounded domain in $\mathbb{C}^2$ and $\partial \Omega$ is $\mathcal{C}^\infty$-smooth, pseudoconvex and of D'Angelo finite type near $\xi_0\in \partial \Omega$. If $\{\eta_j\}\subset \Omega $ is a sequence converging spherically $\frac{1}{2m}$-tangentially to  $\xi_0 \in \partial \Omega$ (cf. Definition \ref{spherically-convergence}), then we have
\begin{align*}
K_{\Omega}(\eta_j, \eta_j) &\approx \frac{1}{\tau_j^2 \epsilon_j^2}\text{ with } \epsilon_j\approx d_\Omega(\eta_j),\tau_{j}:=|\alpha_{j}|.\Big(\dfrac{\epsilon_j}{|\alpha_{j}|^{2m}}\Big)^{1/2};\\
d^2_\Omega(\eta_j;\xi)&\approx  \frac{|\xi_2|^2}{\epsilon_j^2}+\max\{\ell_j,1\}~\frac{ |\xi_1|^2}{\tau_j^2}\text{ with } \ell_{j}\approx \left(\epsilon_j^{-1}\tau_{j}\left|\frac{\partial\rho(\eta_j)}{\partial z}\right|\right)^2 ;\\
\lim_{j\to\infty} \mathrm{Sec}_{\Omega}(\eta_j;\xi) &= -\frac{4}{3},\lim_{j\to\infty} \mathrm{Ric}_{\Omega}(\eta_j;\xi) =-1, \lim_{j\to\infty} \mathrm{Scal}_{\Omega}(\eta_j)=-2,
\end{align*}
where $K_{\Omega}(z, z), d^2_\Omega(\eta_j;\xi),\mathrm{Sec}_{\Omega}(z,\xi),\mathrm{Ric}_{\Omega}(z,\xi)$, and $\mathrm{Scal}_{\Omega}(z,\xi) $ respectively denote the Bergman kernel, the Bergman metric, the holomorphic sectional curvature, the Ricci curvature, and the scalar curvature of $\Omega$ at $z$ in the direction $\displaystyle\xi=\xi_1\dfrac{\partial}{\partial z}+\xi_{2}\dfrac{\partial}{\partial w}\in T^{1,0}_{z} \Omega\setminus \{0\}$.
\end{theorem}

We notice that the case that the sequence $\{\eta_j\}$ does not satisfy the $(B,\xi_0)$-condition (cf. Definition \ref{eq:balanced-condition}), such as $\frac{\partial\rho(\eta_j)}{\partial z}=0$ given in Example \ref{non-balance}, may occur. However, in general $\{\eta_j\}$ satisfies the $(B,\xi_0)$-condition by virtue of tangential convergences. Namely, we also have the following corollary.
\begin{corollary}\label{cor:higher-dim-bergman-2}
Under the same hypotheses as in Theorem~\ref{holo-sec-curvature-two}, assume also that  $\{\eta_j\}\subset \Omega $ satisfies the $(B,\xi_0)$-condition (cf. Definition \ref{eq:balanced-condition}). Then the Bergman metric admits the asymptotic expansion
\begin{equation}\label{eq:higher-dim-metric}
d^2_\Omega(\eta_j;\xi)\approx \frac{|\xi_{2}|^2}{\epsilon_j^2} + \ell_j \frac{|\xi_1|^2}{\tau_{j}^2}
\end{equation}
for all $\displaystyle\xi=\xi_1\dfrac{\partial}{\partial z}+\xi_{2}\dfrac{\partial}{\partial w}\in T^{1,0}_{\eta_j} \Omega\setminus \{0\}$, where $\epsilon_j\approx d_\Omega(\eta_j)$ and   $\displaystyle \ell_j := \frac{|\alpha_{j}|^{2m}}{\epsilon_j} \to +\infty$ as $j \to \infty$.
\end{corollary}

 Based on the H\"ormander weighted $L^2$-estimates \cite{Ho65} and the Pinchuk scaling method \cite{Pi91}, D. Catlin \cite{Cat89} and F. Berteloot \cite{Be03, Be25} proved that the Kobayashi metric, the Carath\'eodory metric, the Bergman metric of $\Omega$ at $\eta_j$ are all equivalent to 
\begin{align*}
M_\Omega(\eta_j,X) := \|F_{\eta_j}'(\eta_j) X\|
\end{align*}
on $U_0$, where $\|\cdot\|$ is a norm on $\mathbb C^2$ and $\{F_j\}\subset \mathrm{Aut}(\mathbb{C}^2)$ is a suitable scaling sequence such that $F_j(\Omega)$ converges normally to the above-mentioned model $M_P$. In addition, the estimates for the Bergman kernel function and associated curvatures were established in \cite{Cat89, Mc89a, Mc89b}, determined by the boundary behavior of $\{\eta_j\}$. When $\{\eta_j\}$ converges notangentially (or even $\Big(\dfrac{1}{2m}\Big)$-nontangentially in the sense of \cite{NN20}) to $\xi_0$, these estimates are exactly those given in \cite{BSY95, KrY96} restricted to the two-dimensional case. However, in the case when $\{\eta_j\}$ converges spherically $\dfrac{1}{2m}$-tangentially to $\xi_0$ Theorem~\ref{holo-sec-curvature-two} and Corollary~\ref{cor:higher-dim-bergman-2} give a detailed and explicit description for these estimates. 

The organization of this paper is as follows. In Section \ref{technical-section}, we recall basic definitions and results needed later. In Section \ref{S-h-extendible}, we prove Theorem \ref{holo-sec-curvature} and Corollary \ref{cor:higher-dim-bergman}. Finally, the proofs of Theorem~\ref{holo-sec-curvature-two} and Corollary~\ref{cor:higher-dim-bergman-2} is given in Section \ref{S4}.

\section{Preliminaries}\label{technical-section}

\subsection{Normal convergence} 
Let us recall the following definition (see \cite{GK87, Kr21}, or \cite{DN09}). 
\begin{define} Let $\{\Omega_j\}_{j=1}^\infty$ be a sequence of domains in $\mathbb C^n$. We say that $\{\Omega_j\}_{j=1}^\infty$ \emph{converges normally} to a domain $\Omega_0\subset \mathbb C^n$ if the following two conditions hold:
	\begin{enumerate}
		\item[(i)] If a compact set $K$ is contained in the interior of  $\displaystyle\bigcap_{j\geq j_0} \Omega_j$ for some $j_0\in \mathbb N_{\geq 1}$, then $K\subset \Omega_0$.
		\item[(ii)] If a compact subset $K'\subset \Omega_0$, then there exists $j_0\in \mathbb N_{\geq 1}$ such that $\displaystyle K'\subset \bigcap_{j\geq j_0} \Omega_j$.
	\end{enumerate}  
In addition, if a sequence of maps $f_j\colon D_j\to\mathbb C^k$ converges uniformly on compact sets to a map $\varphi_j\colon D\to\mathbb C^m$ then we say that $\varphi_j$ \emph{converges normally} to $\varphi$. 
\end{define}

\subsection{Catlin's multitype} 
In this subsection, we recall the \emph{Catlin's multitype} (cf. \cite{Cat84}). Let $\Omega$ be a domain in $\mathbb C^n$ and $\rho$ be a defining function for $\Omega$ near $p\in \partial\Omega$. Denote by $\Gamma^n$ the set of all $n$-tuples of numbers $\mu=(\mu_1,\ldots,\mu_n)$ such that
\begin{itemize}
\item[(i)] $1\leq \mu_1\leq \cdots\leq\mu_n\leq +\infty$;
\item[(ii)] For each $j$, either $\mu_j=+\infty$ or there is a set of non-negative integers $k_1,\ldots,k_j$ with $k_j>0$ such that
\[
\sum_{s=1}^j \frac{k_s}{\mu_s}=1.
\]
\end{itemize}

A weight $\mu\in \Gamma^n$ is called \emph{distinguished} if there are holomorphic coordinates $(z_1,\ldots,z_n)$ about $p$ with $p$ maps to the origin such that 
\[
D^\alpha \overline{D}^\beta \rho(p)=0~\text{whenever}~\sum_{i=1}^n\frac{\alpha_i+\beta_i}{\mu_i}<1.
\]
Here and in what follows, $D^\alpha$ and $\overline{D}^\beta$ denote the partial differential operators
\[
\frac{\partial^{|\alpha|}}{\partial z_1^{\alpha_1}\cdots \partial z_n^{\alpha_n }}~\text{and}~\frac{\partial^{|\beta|}}{\partial \bar z_1^{\beta_1}\cdots \partial \bar z_n^{\beta_n }},
\]
respectively. 
\begin{define}
The \emph{multitype} $\mathcal{M}(z_0)$ is defined to be the smallest weight $\mathcal{M}=
(m_1,\ldots,m_n)$ in $\Gamma^n$~(smallest in the lexicographic sense) such that $\mathcal{M}\geq \mu$ for every distinguished weight $\mu$.
\end{define}
 \subsection{The $h$-extendibility }
A multiindex $(\lambda_1,\lambda_2, \ldots, \lambda_n)$ is called a \emph{multiweight} if $1\geq \lambda_1\geq \cdots\geq \lambda_n$. Now let us recall the following definitions (cf. \cite{Yu94, Yu95}).  
\begin{define} \label{def-28} Let $\Lambda=(\lambda_1,\lambda_2, \ldots, \lambda_n)$ be a multiweight and let us define
$$
\sigma(z)=\sigma_\Lambda(z):=\sum_{j=1}^n |z_j|^{1/\lambda_j}.
$$
One says that a function $f\colon \mathbb  C^n\to \mathbb R$ is \emph{$\Lambda$-homogeneous with weight $\alpha$} if 
$$
f\big(t^{\lambda_1}z_1,t^{\lambda_2}z_2, \ldots,t^{\lambda_n}z_n\big)=t^\alpha f(z), \; \forall t\geq 0, z\in \mathbb C^n.
$$
 In case $\alpha=1$, then $f$ is simply called \emph{$\Lambda$-homogeneous}. For example, the function $\sigma_\Lambda$ is $\Lambda$-homogeneous. In addition, for a multiweight $\Lambda$ and a real-valued $\Lambda$-homogeneous function $P$, we define a homogeneous model $D_{\Lambda,P}$ as follows:
 $$
D_{\Lambda,P}=\left\{ (z,w)\in \mathbb C^n\times \mathbb C\colon \mathrm{Re}(w)+ P(z)<0 \right\}.
 $$
\end{define} 
 \begin{define} Let $D_{\Lambda,P}$ be a homogeneous model. Then $D_{\Lambda,P}$ is called \emph{$h$-extendible} if there exists a $\Lambda$-homogeneous $\mathcal{C}^1$ function $a(z)$ on $\mathbb C^n\setminus\{0\}$ satisfying the following conditions:
 \begin{itemize}
 \item[(i)] $a(z)>0$ whenever $z\ne 0$;
 \item[(ii)] $P(z)-a(z)$ is plurisubharmonic on $\mathbb C^n$. 
  \end{itemize}
We will call $a(z)$ a \emph{bumping function}. 
 \end{define}

By a pointed domain $(\Omega,p)$ in $\mathbb C^{n+1}$ one means that $\Omega$ is a smooth pseudoconvex domain in $\mathbb C^{n+1}$ with $p\in \partial\Omega$. Let $\rho$ be a local defining function for $\Omega$ near $p$ and let the multitype $\mathcal{M}(p)=(2m_1,\ldots,2m_n, 1)$ be finite. We note that because of pseudoconvexity, the integers $2m_1,\ldots,2m_n$ are all even. Then, by definition, there are distinguished coordinates $(z,w)=(z_1,\ldots,z_n,w)$ such that $p=(0',0)$ and $\rho(z,w)$ can be expanded near $(0',0)$ as follows:
$$
\rho(z,w)=\mathrm{Re}(w)+P(z)+R(z,w),
$$ 
 where $P$ is a $(1/2m_1,\ldots,1/2m_n)$-homogeneous plurisubharmonic polynomial that contains no pluriharmonic terms, $R$ is smooth and satisfies 
 $$
 |R(z,w)|\lesssim\left( |w|+ \sum_{j=1}^n |z_j|^{2m_j} \right)^\gamma,
 $$ 
 for some constant $\gamma>1$. 
 In what follows, we assign weights $\frac{1}{2m_1}, \ldots,\frac{1}{2m_{n}}, 1$ to the variables $z_1,\ldots, z_{n}, w$, respectively and denote by $wt(K):=\sum_{j=1}^{n} \frac{k_j}{2m_j}$ the weight of an $n$-tuple $K=(k_1, \ldots, k_{n})\in \mathbb Z^{n}_{\geq0}$. Notice that  $wt(K+L)=wt(K)+wt(L)$ for any $K, L\in \mathbb Z^{n}_{\geq0}$.
\begin{define}\label{def-order} We say that $M_P=\{(z,w)\in \mathbb C^n\times \mathbb C\colon \mathrm{Re}(w)+P(z)<0\}$ is an \emph{associated model} for $(\Omega,p)$. If the pointed domain $(\Omega,p)$ has an $h$-extendible  associated model, we say that $(\Omega,p)$ is \emph{$h$-extendible}. 
\end{define}

Next, we recall the following definition (cf. \cite{Yu95}).
\begin{define} \label{vanishing-order}
Let $\Lambda=(\lambda_1,\ldots,\lambda_n)$ be a fixed $n$-tuple of positive numbers and $\mu>0$. We denote by $\mathcal{O}(\mu,\Lambda)$ the set of smooth functions $f$ defined near the origin of $\mathbb C^n$ such that
$$
D^\alpha \overline{D}^\beta f(0)=0~\text{whenever}~ \sum_{j=1}^n (\alpha_j+\beta_j)\lambda_j \leq \mu.
$$
In addition, we use $\mathcal{O}(\mu)$ to denote the functions of one variable, defined near the origin of $\mathbb C$,  vanishing to order at least $\mu$ at the origin. 
\end{define}
\subsection{The Bergman kernel, the Bergman metric, and its curvatures}
Let $\Omega$ be a bounded domain in $\mathbb{C}^n$. Let us define the \emph{Bergman space}
\[
A^2(\Omega) := L^2(\Omega) \cap H(\Omega),
\]
where $H(\Omega)$ is the space of holomorphic functions on $\Omega$ and $L^2(\Omega)$ is the space of square integrable functions on $\Omega$. It is well-known that $A^2(\Omega)$ is a Hilbert space and let $\{\phi_j\}_{j=0}^\infty$ be a complete orthonormal basis for $A^2(\Omega)$. Then the \emph{Bergman kernel} and \emph{Bergman metric} at $z\in \Omega$ along the direction $\displaystyle X = \sum_{i=1}^n X_i \frac{\partial}{\partial z_i} \in T_z^{1,0}(\Omega)$  are, respectively, defined by
\begin{align*}
K_\Omega(z, \bar{z}) &:= \sum_{j=0}^\infty \phi_j(z)\overline{\phi_j(z)};\\
d^2_\Omega(z; X) &:=\sum_{j,k=1}^n g_{j\bar{k}} X_j \overline{X_k},
\end{align*}
where $\displaystyle g_{j\bar{k}}= \frac{\partial^2 \log K_\Omega(z, \bar{z})}{\partial z_j \partial \bar{z}_k}$ for $1\leq i,k\leq n$. Moreover, the \emph{bisectional curvature} $B_\Omega(z; X, Y)$ at $z$ along the directions $X$ and $Y$ is given by
\[
B_\Omega(z; X, Y) = \frac{R_{h\bar{j}k\bar{l}} X_h \overline{X_j} Y_k \overline{Y_l}}{g_{j\bar{k}} X_j \overline{X_k} g_{l\bar{m}} Y_l \overline{Y_m}},
\]
where
\[
R_{h\bar{j}k\bar{l}} = -\frac{\partial^2 g_{j\bar{h}}}{\partial z_k \partial \bar{z}_l} + g^{\nu\bar{\mu}} \frac{\partial g_{j\bar{\mu}}}{\partial z_k} \frac{\partial g_{\nu\bar{h}}}{\partial \bar{z}_l}.
\]
Here, we use the Einstein convention and $g^{\nu\bar{\mu}}$ denotes the components of the inverse matrix of $(g_{j\bar{k}}) $. Then, the \emph{holomorphic sectional curvature} $\mathrm{Sec}_\Omega(z; X)$ and \emph{Ricci curvature} $\mathrm{Ric}_\Omega(z; X)$, and  the scalar curvature $\mathrm{Scal}_\Omega(z)$ at $z$ along the direction $X$ are, respectively, defined by
\begin{align*}
\mathrm{Sec}_\Omega(z; X) &= B_\Omega(z; X, X);\\
 \mathrm{Ric}_\Omega(z; X) &= \sum_{j=1}^n B_\Omega(z; E_j, X);\\
  \mathrm{Scal}_\Omega(z) &= \sum_{hjkl} g^{j\bar{h}}(z) g^{kl}(z) R_{hjkl}(z), 
\end{align*}
where $\{E_1, \ldots, E_n\}$ is a basis of $T_z^{1,0}(\Omega)$.

\subsection{The minimum integrals}

Let $\Omega$ be a bounded domain in $\mathbb{C}^n$. For $z \in \Omega$ and $\displaystyle X = \sum_{i=1}^n X_i \frac{\partial}{\partial z_i} \in T_z^{1,0}(\Omega)\setminus\{0\}$, the minimum integrals are defined as follows:
\begin{align*}
I_0^\Omega(z) &= \inf\Big\{\int_\Omega |f|^2 d\mu \colon f \in A^2(\Omega), f(z) = 1\Big\}; \\
I_1^\Omega(z;X) &= \inf\Big\{\int_\Omega |f|^2 d\mu \colon f \in A^2(\Omega), f(z) = 0, \sum_{j=1}^n X_j \frac{\partial f}{\partial z_j}(z) = 1\Big\} ;\\
I_2^\Omega(z;X) &= \inf\left\{\int_\Omega|f|^2 d\mu \colon f \in A^2(\Omega), f(z) = \frac{\partial f}{\partial z_1}(z) = \cdots = \frac{\partial f}{\partial z_n}(z) = 0,\right. \\
&\qquad\left. \sum_{j,k=1}^n X_j X_k \frac{\partial^2 f}{\partial z_j \partial z_k}(z) = 1\right\}.
\end{align*}

Then, we recall the following formulas (cf. \cite{Ber33, Ber70, Fu37}).
\begin{align*}
K_\Omega(z, \bar z) &= \frac{1}{I_0^\Omega(z)}; \\
d^2_\Omega(z; X) &= \frac{I_0^\Omega(z)}{I_1^\Omega(z; X)}; \\
R_\Omega(z, X) &= 2 - \frac{[I_1^\Omega(z; X)]^2}{I_0^\Omega(z) I_2^\Omega(z; X)}.
\end{align*}

We now prove the following lemma for localization of minimum integrals, which allows us to localize the holomorphic sectional curvature of the Bergman metric.

\begin{lemma}\label{local-approx}
Let $D$ be a bounded pseudoconvex domain in $\mathbb{C}^{n+1}$ with $C^\infty$-smooth boundary and let $\xi_0\in\partial D$ be an $h$-extendible boundary point. Let $U$ be a neighborhood of $\xi_0$ and let $\{\eta_j\}\subset D$ be a sequence converging to $\xi_0 \in \partial D$. Then, for $i = 0, 1, 2$, we have
$$\lim\limits_{j\to\infty}\frac{I_i^D(\eta_j; \xi)}{I_i^{D \cap U}(\eta_j; \xi)}=1,\quad \forall \xi\in \mathbb{C}^{n+1}\setminus \{0\}.
$$
\end{lemma}

\begin{proof} 
By Theorem 4.1 in \cite{Yu94}, there exists a local holomorphic peak function $h$ for $D$ at $\xi_0$. Let $V\Subset U$ be a neighborhood of $\xi_0$ such that $|h(z)|\leq a<1$ on $\overline{D}\setminus V$. Therefore, by \cite[Lemma 1]{KS02} (see also \cite[Theorem 4]{KY96}), one obtains
\begin{align*}
1\leq \frac{I_i^D(\zeta; \xi)}{I_i^{D \cap U}(\zeta; \xi)} 
\leq  \frac{(1 + ca^k)^2}{|h(\zeta)|^{2k}}, \quad i=0,1,2, \quad \forall \zeta\in V\cap D,\quad \forall \xi\in \mathbb{C}^{n+1}\setminus \{0\}.
\end{align*}

Since $h(\eta_j)\to h(\xi_0)=1$ as $j\to \infty$, we may assume that $1-1/n_j< |h(\eta_j)|< 1$ for some sequence $\{n_j\}\subset \mathbb{N}$. If we let $k_j=\lfloor\sqrt{n_j}\rfloor$ for all $j\in \mathbb{N}$, then 
$|h(\eta_j)|^{2k_j}\to 1$ as $j\to \infty$. Moreover, $a^{k_j}\to 0$ as $j\to \infty$. Hence, we conclude that
$$
\lim_{j\to\infty} \frac{I_i^{D}(\eta_j; \xi)}{I_i^{D \cap U}(\eta_j; \xi)} = 1,\quad i=0,1,2, \quad \forall \xi\in \mathbb{C}^{n+1}\setminus \{0\},
$$
and the proof is complete.
\end{proof}

\subsection{The boundary behavior of the Bergman kernel function, the Bergman metric, and the associated curvatures}
In this subsection,  we recall the following results. First of all, the following theorem ensures the stability of the Bergman kernel (see \cite{KY96, KK03}).
\begin{theorem} [See Proposition in \cite{KY96} or Theorem $3.7$ in \cite{KK03}] \label{approx-holo-cur}
Let $D$ be a bounded domain in $\mathbb{C}^n$ containing the origin $0$. Let $D_j$ denote a sequence of bounded domains in $\mathbb{C}^n$ that converges to $D$ in $\mathbb{C}^n$ in the sense that, for every $\epsilon > 0$, there exists $N > 0$ such that $(1 - \epsilon)D \subset D_j \subset (1 + \epsilon)D$ for every $j > N$. Then, for every compact subset $F$ of $D$, the sequence of Bergman kernel functions $K_{D_j}$ converges uniformly to $K_D$ on $F \times F$.
\end{theorem}

Next, by virtue of the Cauchy estimates on the Bergman kernel functions, the derivatives of the Bergman kernels also converge uniformly on compacta of $D$. Therefore, we have the following corollary (cf.  \cite{KY96, KK03}).
\begin{corollary}\label{cor-approx-holo-cur}
Let $D$ be a bounded domain in $\mathbb{C}^n$ containing the origin 0. Let $D_j$ denote a sequence of bounded domains in $\mathbb{C}^n$ that converges to $D$ in $\mathbb{C}^n$ in the sense that, for every $\epsilon > 0$, there exists $N > 0$ such that $(1 - \epsilon)D \subset D_j \subset (1 + \epsilon)D$ for every $j > N$. Then, for every compact subset $F$ of $D$, we have
\begin{itemize}
  \item[(i)] $d^2_{D_j}(p; X)$ converges uniformly to $d^2_D(p; X)$ on $F \times \mathbb C^n$;
   \item[(ii)] $\mathrm{Sec}_{D_j}(p; X)$ converges uniformly to $\mathrm{Sec}_D(p; X) $ on $F \times \mathbb C^n$;
 \item[(iii)] $\mathrm{Ric}_{D_j}(p; X)$ converges uniformly to $\mathrm{Ric}_D(p; X) $ on $F \times \mathbb C^n$;
\item[(iv)] $\mathrm{Scal}_{D_j}(p)$ converges uniformly to $\mathrm{Scal}_D(p) $ on $F \times \mathbb C^n$.
 \end{itemize}
\end{corollary}

Finally, in the case when $D$ is the unit ball $\mathbb B^n$, by the above corollary and \cite[Theorem $3.1$ and Theorem $4.4$ ]{Zh17} we obtain the following corollary.
\begin{corollary}\label{cor-approx-holo-cur-2}
 Let $D_j$ denote a sequence of bounded domains in $\mathbb{C}^n$ that converges to $\mathbb B^n$ in $\mathbb{C}^n$ in sense that, for every $\epsilon > 0$, there exists $N > 0$ such that $(1 - \epsilon)\mathbb B^n \subset D_j \subset (1 + \epsilon)\mathbb B^n$ for every $j > N$. Then, for any $X\in \mathbb C^{n}\setminus\{0\}$, we have
\begin{itemize}
   \item[(i)] $\displaystyle\lim_{j\to\infty}\mathrm{Sec}_{D_j}(0; X)=-\frac{4}{n} $;
    \item[(ii)] $\displaystyle\lim_{j\to\infty}\mathrm{Ric}_{D_j}(0; X)=-1 $;
 \item[(iii)] $\displaystyle\lim_{j\to\infty}\mathrm{Scal}_{D_j}(0)=-n $.
  \end{itemize}
\end{corollary}

\section{The boundary behavior of the Bergman kernel, the Bergman metric, and curvatures near a strongly $h$-extendible point}\label{S-h-extendible}
\subsection{$\Lambda$-tangential convergence}\label{Ss3.1}

Throughout this subsection, let $\Omega$ be a domain in $\mathbb C^{n+1}$ and let $\xi_0\in \partial \Omega $ be an $h$-extendible boundary point \cite{Yu95} (or, semiregular point in the terminology of \cite{DH94}). Let $\mathcal{M}(\xi_0)=(2m_1,\ldots,2m_n,1)$ be the finite multitype of  $\partial \Omega$ at $\xi_0$ (see \cite{Cat84}) and denote by $\Lambda=(1/2m_1,\ldots,1/2m_n)$. By following the proofs of Lemmas $4.10$, $4.11$ in \cite{Yu95}, after a change of variables there are the coordinate functions $(z,w)=(z_1,\ldots, z_n,w)$ such that $\xi_0=(0',0)$ and  $\rho(z,w)$, the local defining function for $\Omega$ near $\xi_0$, can be expanded near $(0',0)$ as follows:
$$
\rho(z,w)=\mathrm{Re}(w)+ P(z) +R_1(z) + R_2(\mathrm{Im} w)+(\mathrm{Im} w) R(z),
 $$ 
 where $P$ is a $\Lambda$-homogeneous plurisubharmonic polynomial that contains no pluriharmonic monomials, $R_1\in \mathcal{O}(1, \Lambda),R\in \mathcal{O}(1/2, \Lambda) $, and $R_2\in \mathcal{O}(2)$ (cf. Definition \ref{vanishing-order}). 

In what follows, let us recall that $d_\Omega(z)$ denotes the Euclidean distance from $z$ to $\partial\Omega$. We now recall the following definition.
\begin{define}[See Definition $3.1$ in \cite{NNN25}]\label{lambda-tangent}
We say that a sequence $\{\eta_j=(\alpha_j,\beta_j)\}\subset  \Omega$ with $\alpha_j=(\alpha_{j 1},\ldots,\alpha_{j n})$, \emph{converges uniformly $\Lambda$-tangentially to $\xi_0$} if the following conditions hold:
\begin{itemize}
\item[(a)] $|\mathrm{Im}(\beta_j)|\lesssim |d_\Omega(\eta_j)|$;
\item[(b)] $|d_\Omega(\eta_j)|=o(|\alpha_{jk}|^{2m_k})$ for $1\leq k\leq n$;
\item[(c)] $|\alpha_{j1}|^{2m_1}\approx |\alpha_{j2}|^{2m_2}\approx \cdots\approx |\alpha_{jn}|^{2m_n}$,
\end{itemize}

\end{define}
\begin{remark}
According to \cite{NN20}, $\{\eta_j\}\subset \Omega$ converges $\Lambda$-nontangentially to $\xi_0$ if $|\mathrm{Im}(\beta_j)|\lesssim |d_\Omega(\eta_j)|$ and $|\alpha_{j k}|^{2m_k}\lesssim|d_\Omega(\eta_j)|$ for every $1\leq k\leq n$. Therefore,  the uniformly $\Lambda$-tangential convergence is a type of $\Lambda$-tangential convergences.
\end{remark}

It is well-known that Euler's identity for weighted homogeneous polynomials gives
$$
2\mathrm{Re} \sum_{j=1}^{n} \frac{\partial P}{\partial z_j} \frac{z_j}{2m_j} = P(z)
$$
 for all $z \in \mathbb{C}^{n}$ (cf. \cite[Lemma $2$]{NNTK19}). However, we need the following condition to ensure that all tangential directions behave uniformly near $\xi_0$. 
\begin{define}\label{eq:balanced-condition}
We say that a sequence $\{\eta_j=(\alpha_j,\beta_j)\}\subset  \Omega$ satisfies the \emph{balanced condition}, say the $(B,\xi_0)$-condition, if
\begin{equation*}
\left|\alpha_{j1}\frac{\partial P(\alpha_j)}{\partial z_1}\right| \approx \left|\alpha_{j2}\frac{\partial P(\alpha_j)}{\partial z_2}\right| \approx \cdots \approx \left|\alpha_{j,n}\frac{\partial P(\alpha_j)}{\partial z_{n}}\right| \approx |\alpha_{j1}|^{2m_1}\approx \cdots\approx |\alpha_{jn}|^{2m_n},
\end{equation*}
\end{define}

Now let us denote by $\displaystyle \sigma(z):=\sum_{k=1}^n |z_k|^{2m_k}$ and recall the following definition.
\begin{define}[See Definition $3.2$ in \cite{NNN25}]\label{strongly-h-extendible} We say that a boundary point $\xi_0\in \partial \Omega$ is \emph{strongly $h$-extendible} if there exists $\delta>0$ such that $P(z)-\delta \sigma(z)$ is plurisubharmonic, i.e. $dd^c P\geq \delta dd^c \sigma$. 
\end{define}
\begin{remark}\label{remark-strongly-h-extendible} 
Since $dd^c P\gtrsim  dd^c \sigma$, it follows that
\begin{align*}
\sum_{k, l=1}^n \frac{\partial^2P}{\partial z_k\partial \bar z_l} (\alpha) w_j\bar w_l&\gtrsim \sum_{k, l=1}^n \frac{\partial^2\sigma}{\partial z_k\partial \bar z_l} (\alpha) w_j\bar w_l\\
&\gtrsim m_1^2|\alpha_1|^{2m_1-2}|w_1|^2+\cdots+m_n^2|\alpha_n|^{2m_n-2}|w_n|^2
\end{align*}
for all $\alpha,w\in \mathbb C^n$. This implies that $P$ is strictly plurisubharmonic away from the union of all coordinates axes, i.e. $M_P$ is \emph{homogeneous finite diagonal type} in the sense of  \cite{He92b, He16} (or $M_P$ is a \emph{$WB$-domain} in the sense of  \cite{AGK16}).
\end{remark}

From now on, we assume that $\xi_0\in \partial\Omega$ is a strongly $h$-extendible point. For a given sequence $\{\epsilon_j\}\subset \mathbb{R}^+$, we define the corresponding sequence  $\tau_j=(\tau_{j1},\ldots,\tau_{jn})$ by
\begin{equation*}
\tau_{jk}:=|\alpha_{jk}|\left(\frac{\epsilon_j}{|\alpha_{jk}|^{2m_k}}\right)^{1/2}, \quad j\geq 1, \; 1\leq k\leq n.
\end{equation*}
Then, a direct computation yields that $\tau_{jk}^{2m_k}=\epsilon_j\left(\frac{\epsilon_j}{|\alpha_{jk}|^{2m_k}}\right)^{m_k-1}\lesssim \epsilon_j$. Consequently, we have
$$
\epsilon_j^{1/2}\lesssim \tau_{jk}\lesssim \epsilon_j^{1/2m_k}.
$$

To close this subsection, we recall the following lemma (see a proof in \cite{NNN25}). 
 \begin{lemma}[See Lemma $3.2$ in \cite{NNN25}]\label{Cn-spherical-convergence} If $P(z)-\delta \sigma(z)$ is plurisubharmonic for some $\delta>0$, then
\begin{align*}
\epsilon_j^{-1}\sum_{k, l=1}^n \frac{\partial^2P}{\partial z_k\partial \bar z_l} (\alpha_j)\tau_{jk}\tau_{jl} w_k\bar w_l\gtrsim m_1^2|w_1|^2+\cdots+m_n^2|w_n|^2.
\end{align*}
\end{lemma}

\subsection{Estimates of Bergman kernel function and associated invariants near a strongly $h$-extendible point}

In this subsection, we shall prove Theorem \ref{holo-sec-curvature} and Corollary \ref{cor:higher-dim-bergman}. We also provide an illustrative example.
\begin{proof}[Proof of Theorem \ref{holo-sec-curvature}]
Let $\Omega$ and $\xi_0\in \partial \Omega $ be as in the statement of Theorem \ref{holo-sec-curvature}. As in Subsection \ref{Ss3.1}, there exist local coordinates $(z, w)=(z_1,\ldots,z_n,w)$ near $\xi_0$ such that $\xi_0=(0',0)$ and the local defining function $\rho(z,w)$ for $\Omega$ near $(0',0)$ is described as follows:
$$
\rho(z,w)=\mathrm{Re}(w)+ P(z) +R_1(z) + R_2(\mathrm{Im} w)+(\mathrm{Im} w) R(z),
$$ 
 where $P$ is a $\Lambda$-homogeneous plurisubharmonic polynomial that contains no pluriharmonic monomials, $R_1\in \mathcal{O}(1, \Lambda),R\in \mathcal{O}(1/2, \Lambda) $, and $R_2\in \mathcal{O}(2)$.

By assumption, the sequence $\eta_j=(\alpha_j,\beta_j)=(\alpha_{j1},\ldots,\alpha_{jn},\beta_j)$ converges uniformly $\Lambda$-tangentially to $\xi_0$, i.e.,
\begin{itemize}
\item[(a)] $|\mathrm{Im}(\beta_j)|\lesssim |d_\Omega(\eta_j)|$;
\item[(b)] $|d_\Omega(\eta_j)|=o(|\alpha_{jk}|^{2m_k})$ for $1\leq k\leq n$;
\item[(c)] $|\alpha_{j1}|^{2m_1}\approx |\alpha_{j2}|^{2m_2}\approx \cdots\approx |\alpha_{jn}|^{2m_n}$.
\end{itemize}

Fix a small neighborhood $U_0$ of the origin. We may assume without loss of generality that the sequence $\{\eta_j=(\alpha_j,\beta_j)\}\subset U_0\cap \Omega$. Writing $\beta_j=a_j+i b_j$ with $\epsilon_j>0$, we define the associated boundary points $\eta_j'=(\alpha_{j}, a_j +\epsilon_j+i b_j)\in\{\rho=0\}$ for each $j\in\mathbb{N}_{\geq 1}$. Note that $\epsilon_j\approx d_\Omega(\eta_j)$.

We employ the scaling technique. Following the approach in the proof of Theorem $1.1$ in \cite{NNN25}, we perform several sequences of coordinate transformations. Let us first consider the sequences of translations $L_{\eta_j'}\colon \mathbb C^{n+1}\to\mathbb C^{n+1}$, defined by
$$
(\tilde z,\tilde w)=L_{\eta_j'}(z,w):=(z,w)-\eta_j'=(z-\alpha_j,w-\beta_j').
$$

Next, we define the sequence $\{Q_j\}$ of polynomial automorphisms of $\mathbb C^{n+1}$ by 
\[
\begin{cases}
w:= \tilde{w}+(R_2'(b_j) + R(\alpha_j)) i\tilde w+2\sum\limits_{1\leq |p|\leq 2} \frac{D^pP}{p!} (\alpha_j)(\tilde z)^p+2\sum\limits_{1\leq |p|\leq 2} \frac{D^p R_1}{p!}(\alpha_j)  (\tilde z)^p\\
\hskip 1cm+b_j\sum\limits_{1\leq |p|\leq 2} \frac{D^p R}{p!}(\alpha_j)(\tilde z)^p ;\\
z_k:=\tilde{z}_k,\, k=1,\ldots,n.
\end{cases}
\]

Finally, we introduce an anisotropic dilation $\Delta_j\colon  \mathbb{C}^{n+1}\to\mathbb{C}^{n+1}$, given by
\begin{equation*}\label{dilationj}
\Delta_j(z,w):=\Delta_{\eta_j}^{\epsilon_j} (z_1,\ldots,z_n, w)=\left(\frac{z_1}{\tau_{j1}},\ldots,\frac{z_n}{\tau_{jn}}, \frac{w}{\epsilon_j}\right),
\end{equation*}
where 
\begin{equation*}
\tau_{jk}:=|\alpha_{jk}|\left(\frac{\epsilon_j}{|\alpha_{jk}|^{2m_k}}\right)^{1/2},\quad 1\leq k\leq n.
\end{equation*}

Consequently, the composition $T_j:=\Delta_j\circ Q_j\circ L_{\eta_j'}\in \mathrm{Aut}(\mathbb{C}^{n+1})$ satisfies $T_j(\eta_j')=(0',0)$ and $T_j(\eta_j)=(0',-1-i(R_2'(b_j) + R(\alpha_j)))\to (0',-1)$ as $j\to\infty$. Furthermore, the transformed hypersurface $T_j(\{\rho=0\})$ admits the defining equation
\begin{align}\label{taylor-defining-function}
\begin{split}
&\epsilon_j^{-1}\rho\left (T_j^{-1}(\tilde{z},\tilde{w})\right)\\
&= \mathrm{Re} (\tilde{w})+\epsilon_j^{-1}o(\epsilon_j|\mathrm{Im}(\tilde{w})|)+\frac{1}{2}\sum_{k,l=1}^n \frac{\partial^2 P}{\partial \tilde{z}_k\partial\overline{\tilde{z}_l}}(\alpha_j) \epsilon_j^{-1}\tau_{jk}\tau_{jl} \tilde{z}_k\overline{\tilde{z}_l}\\
&\quad+\frac{1}{2}\sum_{k,l=1}^n \frac{\partial^2 R_1}{\partial \tilde{z}_k\partial\overline{\tilde{z}_l}}(\alpha_j) \epsilon_j^{-1}\tau_{jk}\tau_{jl} \tilde{z}_k\overline{\tilde{z}_l}+ \frac{\epsilon_j^{-1}b_j}{2}\sum_{k,l=1}^n \frac{\partial^2 R}{\partial \tilde{z}_k\partial\overline{\tilde{z}_l}}(\alpha_j) \tau_{jk}\tau_{jl} \tilde{z}_k\overline{\tilde{z}_l}+\cdots=0,
\end{split}
\end{align} 
where the dots denote higher-order terms.

By virtue of the uniform $\Lambda$-tangential convergence of $\{\eta_j\}$ to $\xi_0=(0',0)$, the authors \cite{NNN25} proved that, up to passing to a subsequence, the defining functions in (\ref{taylor-defining-function}) converge uniformly on compact subsets of $\mathbb{C}^{n+1}$ to $\hat{\rho}(\tilde{z},\tilde{w}):=\mathrm{Re}(\tilde{w})+H(\tilde{z})$, where
$$
H(\tilde{z})=\sum_{k,l=1}^n a_{kl}  \tilde{z}_k\overline{\tilde{z}_l}
$$
with coefficients
$$
a_{kl}:=\frac{1}{2}\lim_{j\to\infty}\frac{\partial^2 P}{\partial \tilde{z}_k\partial\overline{\tilde{z}_l}}(\alpha_j) \epsilon_j^{-1}\tau_{jk}\tau_{jl}, \quad 1\leq k,l\leq n.
$$
As a result, the sequence $T_j(U_0\cap \Omega)$ converges normally to the model
$$
M_{H}:=\left \{(\tilde{z},\tilde{w})\in \mathbb{C}^{n+1}\colon \mathrm{Re}(\tilde{w})+H(\tilde{z})<0\right\}.
$$
In addition, we observe that $\Omega_j:=T_j(\Omega)$ converges also normally to $M_H$.

Since $M_H$ is the limit of the pseudoconvex domains $T_j(U_0\cap \Omega)$, it follows that $M_H$ is pseudoconvex, and hence $H$ is plurisubharmonic. Furthermore, it follows immediately from Lemma \ref{Cn-spherical-convergence} that $H$ is positive definite. Therefore, there exists a unitary matrix $U$ such that
\[
U^*AU = D = \mathrm{diag}(\lambda_1, \ldots, \lambda_n),
\]
where $A=(a_{kl})$ and $\lambda_1, \ldots, \lambda_n>0$ are the eigenvalues of the matrix $A$. We denote $\Lambda=(\lambda_1, \ldots, \lambda_n)$. Then, the linear transformation $\Theta$, defined by
\[
\Theta(z,w) = (Uz, w) = \left(\sum_{j=1}^n U_{1j}z_j, \ldots, \sum_{j=1}^n U_{nj}z_j, w\right),
\]
maps $M_{H}$ onto 
$$
M^\Lambda:=\{(z,w)\in \mathbb{C}^{n+1} \colon\mathrm{Re}(w) +\lambda_1|z_1|^2+\lambda_2|z_2|^2+\cdots+\lambda_n|z_n|^2<0\}.
$$

Next, we define the dilation $\Delta^\Lambda\colon \mathbb{C}^{n+1}\to\mathbb{C}^{n+1}$ by
\begin{equation*}
\Delta^\Lambda(z,w):=\left(\sqrt{\lambda_1} z_1, \ldots, \sqrt{\lambda_n} z_n, w\right).
\end{equation*}
This transformation maps $M^\Lambda$ onto the Siegel half-space 
$$
\mathcal{U}_{n+1}:=\{(z,w)\in \mathbb{C}^{n+1} \colon\mathrm{Re}(w) +|z_1|^2+|z_2|^2+\cdots+|z_n|^2<0\}.
$$
Finally, the holomorphic map $\Psi$ defined by
\[
(z,w)\mapsto  \left( \frac{2z_1}{1-w},\ldots, \frac{2z_n}{1-w},  \frac{w+1}{1-w}\right)
\]
is a biholomorphism from $\mathcal{U}_{n+1}$ onto $\mathbb{B}^{n+1}$.

Now let us consider the sequence of biholomorphic maps 
\[
f_j:=\Psi\circ \Delta^\Lambda\circ \Theta\circ\Delta_j\circ Q_j\circ L_{\eta_j'} \colon \Omega \to f_j(\Omega)=\Psi\circ \Delta^\Lambda\circ \Theta(\Omega_j).
\]
Then, one observes that $f_j(\Omega \cap U_0)$ converges normally to $\mathbb{B}^{n+1}$ and $f_j(\partial \Omega \cap U_0)$ converges to $\partial \mathbb{B}^{n+1}$. Moreover, for any $(z,w)\in U_0\cap \Omega$, the point $(z',w'):=\Delta^\Lambda\circ \Theta\circ\Delta_j\circ Q_j\circ L_{\eta_j'}(z,w)$ satisfies
\[
\mathrm{Re}(w')\gtrsim -\dfrac{1}{\epsilon_j}; \, |z_k'|\lesssim \frac{1}{\tau_{jk}},\,1\leq k\leq n.
\]
Furthermore, a computation shows that, for each $j$, the image under $\Psi$ of the domain
\[
\Big\{(z',w')\in \mathcal{U}_{n+1}\colon \mathrm{Re}(w')\lesssim -\dfrac{1}{\sqrt{\epsilon_j}}\Big\}
\]
is given by 
\[
\Big\{(\tilde z, \tilde w) \in \mathbb{B}^{n+1} \colon 
|\tilde w+1| \lesssim \sqrt{\epsilon_j},\; \sum_{k=1}^n |\tilde z_k|^2 \lesssim \sqrt{\epsilon_j}
\Big\}.
\]
This follows from the estimates
\begin{align*}
|\tilde w+1|=\Big|\dfrac{2}{1-w'}\Big|\lesssim \dfrac{1}{|\mathrm{Re}(w')|},\; |\tilde z_k|=\Big|\frac{2z'_k}{1-w'}\Big|\lesssim \frac{2\,\sqrt{|\mathrm{Re}(w')|}}{|\mathrm{Re}(w')|-1}\lesssim \frac{1}{\sqrt{|\mathrm{Re}(w')|}}
\end{align*}
for $1\leq k\leq n$, where $|z'_k|\lesssim \sqrt{|\mathrm{Re}(w')|}$ follows from (\ref{taylor-defining-function}). In addition, since $\Theta(0',-1) =(0',-1)$ and $\Psi(0',-1)=(0',0)$, we have 
\[
f_j(\eta_j) = \Psi\circ \Theta(0', -1 - i(R_2'(b_j) + R(\alpha_j))) \to (0', 0) \quad \text{as } j \to \infty.
\]
This yields that for a sufficiently small $\epsilon > 0$, there exists $j_0 \in \mathbb{N}_{\geq 1}$ such that
\[
B\big((0',0), 1 - \epsilon\big) \subset F_j(U_0\cap \Omega) \subset B\big((0',0), 1 + \epsilon\big), \, \forall j \geq j_0,
\]
where $F_j(\cdot) := f_j(\cdot) - f_j(\eta_j)$ for all $j \geq j_0$.

In the sequel, we estimate the Bergman kernel function, Bergman metric, and associated curvatures of $\Omega$ at $\eta_j$ in the direction $\displaystyle\xi=\sum_{k=1}^n \xi_k\dfrac{\partial}{\partial z_k}+\xi_{n+1}\dfrac{\partial}{\partial w}\in T^{1,0}_{\eta_j} \Omega\setminus \{0\}$. For the sake of simplicity, we denote $w_0 = -1 - i(R_2'(b_j) + R(\alpha_j)) \sim -1$ and $\gamma_j = R_2'(b_j) + R(\alpha_j) \sim 0$. Since $\Delta_j$, $\Delta^\Lambda$, $L_{\eta_j'}$, and $\Theta$ are all linear, we only compute the Jacobian matrices
\begin{align*}
dQ_j\big|_{(0', -\epsilon_j)} &= \begin{pmatrix}
1 & 0 & \cdots & 0 & 0 \\
0 & 1 & \cdots & 0 & 0 \\
\vdots & \vdots & \ddots & \vdots & \vdots \\
0 & 0 & \cdots & 1 & 0 \\
A_{j1} & A_{j2} & \cdots & A_{jn} & 1 + \gamma_j
\end{pmatrix};\\
d\Psi\big|_{(0',w_0)} &= \begin{pmatrix}
\frac{2}{2 + i\gamma_j} & 0 & \cdots & 0 & 0 \\
0 & \frac{2}{2 + i\gamma_j} & \cdots & 0 & 0 \\
\vdots & \vdots & \ddots & \vdots & \vdots \\
0 & 0 & \cdots & \frac{2}{2 + i\gamma_j} & 0 \\
0 & 0 & \cdots & 0 & \frac{2}{(2 + i\gamma_j)^2}
\end{pmatrix},
\end{align*}
where 
\begin{equation}\label{coeff-1}
A_{jk}= 2\frac{\partial P}{\partial z_k}(\alpha_j) + 2\frac{\partial R_1}{\partial z_k}(\alpha_j) + b_j \frac{\partial R}{\partial z_k}(\alpha_j)\approx \frac{\partial \rho}{\partial z_k}(\eta_j) , \quad 1 \leq k \leq n.
\end{equation}
Therefore, we conclude that
\begin{align*}
dF_j(\xi) &= \left(4\sqrt{\lambda_1}\frac{(U\xi')_1}{\tau_{j1}(2 + i\gamma_j)}, \ldots, 4\sqrt{\lambda_n}\frac{(U\xi')_n}{\tau_{jn}(2 + i\gamma_j)}, 2\frac{\xi_{n+1}(1+\gamma_j) + \sum_{k=1}^n A_{jk} \xi_k}{\epsilon_j(2 + i\gamma_j)^2}\right)\\
&\sim \left(2\sqrt{\lambda_1}\frac{(U\xi')_1}{\tau_{j1}}, \ldots, 2\sqrt{\lambda_n}\frac{(U\xi')_n}{\tau_{jn}}, \frac{\xi_{n+1} + \sum_{k=1}^n A_{jk} \xi_k}{2\epsilon_j}\right)
\end{align*}
for $\displaystyle\xi=\sum_{k=1}^n \xi_k\dfrac{\partial}{\partial z_k}+\xi_{n+1}\dfrac{\partial}{\partial w}\in T^{1,0}_{\eta_j} \Omega\setminus \{0\}$.
Moreover, since $B\big((0',0), 1 - \epsilon\big) \subset F_j(U_0\cap \Omega) \subset B\big((0',0), 1 + \epsilon\big)$ for all $j$ large enough, $F_j(\eta_j) = (0', 0)$, and $U$ is a unitary matrix, by Lemma \ref{local-approx} and Corollary ~\ref{cor-approx-holo-cur} it follows that
\begin{equation}\label{Bergman-eq}
\begin{split}
d^2_\Omega(\eta_j;\xi)\sim d^2_{U_0\cap \Omega}(\eta_j;\xi)&\sim \left(g_{\mathbb{B}^{n+1}}(0; dF_j(\xi), dF_j(\xi))\right)^2 = (n+2)|dF_j(\xi)|^2\\
&\sim (n+2) \left[4 \sum_{k=1}^n \lambda_k \frac{|\xi_k|^2}{\tau_{jk}^2} + \frac{\left|\xi_{n+1} + \sum_{k=1}^n A_{jk} \xi_k\right|^2}{4\epsilon_j^2}\right]\\
&\approx \frac{|\xi_{n+1}|^2}{\epsilon_j^2}+\sum_{k=1}^n \max\{\ell_{jk},1\}\; \frac{|\xi_k|^2}{\tau_{jk}^2},
\end{split}
\end{equation}
where $\displaystyle \ell_{jk}:=(\epsilon_j^{-1}\tau_{jk}|A_{jk}|)^2$ for all $j\geq 1$ and $1\leq k\leq n$.

Next, we shall estimate the Bergman kernel function of $\Omega$ at $\eta_j$. Indeed, by the biholomorphic invariance of the Bergman kernel function, we have
\begin{equation*}
K_{U_0\cap \Omega}(\eta_j, \eta_j) = K_{F_j(U_0\cap \Omega)}(F_j(\eta_j), F_j(\eta_j)) |J_{F_j}(\eta_j)|^2, 
\end{equation*}
where $J_{F_j}(\eta_j)$ is holomorphic Jacobian of $F_j$ at $\eta_j$. A computation shows that
\begin{align*}
\det(dL_{\eta_j'}) &= 1, \quad \det(d\Theta) = 1, \\
\det(d\Delta_j) &= \frac{1}{\tau_{j1} \cdots \tau_{jn} \epsilon_j}, \quad \det(d\Delta^\Lambda) = \sqrt{\lambda_1 \cdots \lambda_n}, \\
\det(dQ_j)\big|_{(0', -\epsilon_j)} &= 1 + R_2'(b_j) + R(\alpha_j) \sim 1, \\
\det(d\Psi)\big|_{(0', -1 - i(R_2'(b_j) + R(\alpha_j)))} &= \frac{2^{n+1}}{(2 + i(R_2'(b_j) + R(\alpha_j)))^{n+2}} \sim \frac{1}{2}.
\end{align*}
Thus, we have
\begin{equation*}
\det J_{\mathbb{C}}(F_j) \sim \frac{\sqrt{\lambda_1 \cdots \lambda_n}}{2 \tau_{j1} \cdots \tau_{jn} \epsilon_j}.
\end{equation*}

As $F_j(\eta_j) = 0 = (0', 0)$ and since $B\big((0',0), 1 - \epsilon\big) \subset F_j(U_0\cap \Omega) \subset B\big((0',0), 1 + \epsilon\big)$ for all $j$ large enough, by Lemma \ref{local-approx} and Corollary ~\ref{cor-approx-holo-cur} one obtains
\begin{align*}
K_{\Omega}(\eta_j, \eta_j)\sim K_{U_0\cap \Omega}(\eta_j, \eta_j) &\sim K_{\mathbb{B}^{n+1}}(0, 0) |\det J_{\mathbb{C}}(F_j)|^2 = \frac{1}{\pi^{n+1}} |\det J_{\mathbb{C}}(F_j)|^2\\
&\sim \frac{\lambda_1 \cdots \lambda_n}{4\pi^{n+1} (\tau_{j1} \cdots \tau_{jn})^2 \epsilon_j^2}\approx \frac{1}{ (\tau_{j1} \cdots \tau_{jn})^2 \epsilon_j^2}.
\end{align*}

Finally, by Corollaries ~\ref{cor-approx-holo-cur} and \ref{cor-approx-holo-cur-2}, it follows that 
\begin{align*}
\lim_{j\to\infty} \mathrm{Sec}_\Omega(\eta_j;\xi) &= \lim_{j\to\infty} \mathrm{Sec}_{F_j(\Omega)}(F_j(\eta_j);dF_j(\eta_j)(\xi)) \\
&= \lim_{j\to\infty} \mathrm{Sec}_{F_j(\Omega)}\left((0',0);\frac{dF_j(\eta_j)(\xi)}{|dF_j(\eta_j)(\xi)|}\right) = -\frac{4}{n+2}
\end{align*}
for any $\xi \in T_{\eta_j}^{1,0}\Omega\setminus \{0\}$. Similarly, we also have
$$
\lim_{j\to\infty} \mathrm{Ric}_{\Omega}(\eta_j;\xi) =-1, \lim_{j\to\infty} \mathrm{Scal}_{\Omega}(\eta_j;\xi)=-(n+1).
$$
Thus, the proof of Theorem~\ref{holo-sec-curvature} is thereby complete.
\end{proof}

\begin{proof}[Proof of Corollary \ref{cor:higher-dim-bergman}] 
By assumption, we have
$$
\left|\alpha_{j1}\dfrac{\partial P(\alpha_j)}{\partial z_1}\right|\approx\cdots \approx \left|\alpha_{jn}\dfrac{\partial P(\alpha_j)}{\partial z_n}\right|\approx |\alpha_{j1}|^{2m_1}\approx |\alpha_{j2}|^{2m_2}\approx \cdots\approx |\alpha_{jn}|^{2m_n}.
$$
In addition, since $|b_j|\lesssim \epsilon_j=o( |\alpha_{j1}|^{2m})$,  (\ref{coeff-1}) implies that $A_{jk}\approx \frac{\partial P}{\partial z_k}(\alpha_j)$. Therefore, one has
$$
\ell_{jk} \approx (\epsilon_j^{-1}A_{jk}\tau_{jk})^2\approx \left(\epsilon_j^{-1}A_{jk}|\alpha_{jk}|\Big(\dfrac{\epsilon_j}{|\alpha_{jk}|^{2m_k}}\Big)^{1/2}\right)^2 \approx \frac{|\alpha_{j1}|^{2m_1}}{\epsilon_j} = \ell_j, \quad 1 \leq k \leq n.
$$
Finally, since $\ell_j := \frac{|\alpha_{j1}|^{2m_1}}{\epsilon_j} \to +\infty$ as $j \to \infty$,  (\ref{Bergman-eq}) yields that
\begin{align*}
d^2_\Omega(\eta_j;\xi)\approx \frac{|\xi_{n+1}|^2}{\epsilon_j^2}+\ell_j \sum_{k=1}^n  \frac{|\xi_k|^2}{\tau_{jk}^2},
\end{align*}
as desired.
\end{proof}

\begin{example} \label{ellipsoid-strongly-h-extendible} Let $\mathcal {E}_{1,2,3}$ be the domain in $\mathbb C^{n+1}$ defined by
$$
\mathcal {E}_{1,2,3}:=\left\{(z_1,z_2,w)\in \mathbb C^{3}\colon \rho(z,w):=\mathrm{Re}(w)+ |z_1|^4+|z_2|^6 <0\right\}.
$$
We note that $\mathcal {E}_{1,2,3}$ is biholomorphically equivalent to the ellipsoid 
$$
\mathcal{D}_{{1,2,3}}:=\left\{(z_1,z_2,w)\in \mathbb C^{3}\colon |w|^2+ |z_1|^4+|z_2|^6<1\right\}
$$
 (cf. \cite{BP95, NNTK19}). Moreover, since $P(z_1,z_2)=|z_1|^4+|z_2|^6=\sigma(z_1,z_2)$ it is obvious that the boundary point $\xi_0=(0,0,0)\in \partial \mathcal{E}_{1,2,3}$ is strongly $h$-extendible.

 Now let us define a sequence $\{\eta_j\}\subset \mathcal {E}_{1,2,3}$ by setting $\eta_j=\big(1/j^{1/4}, 1/j^{1/6},-2/j-1/j^2\big)$ for every $j\in \mathbb N_{\geq 1}$. Then $\rho(\eta_j)=-1/j^2\approx -d_{\mathcal {E}_{1,2,3}}(\eta_j)$, $|\eta_{j1}|^4=|\eta_{j2}|^6=1/j$, and thus $d_{\mathcal {E}_{1,2,3}}(\eta_j)=o(\Big|\dfrac{1}{j^{1/4}}\Big|^4)=o(\Big|\dfrac{1}{j^{1/6}}\Big|^6)$. Hence, the sequence $\{\eta_j\}\subset  \mathcal{E}_{1,2,3}$ converges uniformly $\Lambda$-tangentially to $(0,0,0)\in \partial \mathcal{E}_{1,2,3}$, with $\Lambda=\big(\dfrac{1}{4},\dfrac{1}{6}\big)$, and $\eta_j'=\big(1/j^{1/4}, 1/j^{1/6},-2/j\big)\in \partial \Omega$ for every $j\in \mathbb N_{\geq 1}$.

We see that $\rho(\eta_j)=-\frac{1}{j^2}\approx -d_{\mathcal {E}_{1,2,3}}(\eta_j)$. Set
$\epsilon_j=|\rho(\eta_j)|=\frac{1}{j^2}$. 
In addition, we consider a change of variables $(\tilde z,\tilde w):=L_j(z,w)$, i.e.,
\[
\begin{cases}
\displaystyle w-\frac{2}{j}=\tilde w;\\
\displaystyle z_1-\frac{1}{j^{1/4}}= \tilde{z}_1;\\
\displaystyle z_2-\frac{1}{j^{1/6}}=\tilde{z}_2.
\end{cases}
\]
Then, a direct calculation shows that
\begin{equation*}
\begin{split}
&\rho\circ L_j^{-1}(\tilde z_1,\tilde z_2,\tilde w)=
\mathrm{Re}(\tilde w)-\frac{2}{j}+ |\dfrac{1}{j^{1/4}}+\tilde z_1|^4+|\frac{1}{j^{1/6}}+\tilde z_2|^6 \\
                      &= \mathrm{Re}(\tilde w) + \frac{4}{j^{3/4}}\mathrm{Re}(\tilde{z}_1) + \frac{4}{j^{1/2}}|\tilde{z}_1|^2 + \frac{2}{j^{1/2}}\mathrm{Re}(\tilde{z}_1^2)
+ \frac{4}{j^{1/4}} |\tilde z_1|^2 \mathrm{Re}(\tilde z_1)+ |\tilde z_1|^4\\
&+ \frac{6}{j^{5/6}}\mathrm{Re}(\tilde{z}_2)+ \left(\frac{15}{j^{2/3}}|\tilde{z}_2|^2 + \frac{6}{j^{2/3}}\mathrm{Re}(\tilde{z}_2^2)\right)+ \left(\frac{20}{j^{1/2}}\mathrm{Re}(\tilde{z}_2^3) + \frac{60}{j^{1/2}}\mathrm{Re}(\tilde{z}_2)|\tilde{z}_2|^2\right)+\cdots\\
&= \mathrm{Re}(\tilde w)+ \frac{4}{j^{3/4}}\mathrm{Re}(\tilde{z}_1)+ \frac{6}{j^{5/6}}\mathrm{Re}(\tilde{z}_2) + \frac{2}{j^{1/2}}\mathrm{Re}(\tilde{z}_1^2)+ \frac{6}{j^{2/3}}\mathrm{Re}(\tilde{z}_2^2)+ \frac{4}{j^{1/2}}|\tilde{z}_1|^2 +\frac{15}{j^{2/3}}|\tilde{z}_2|^2\\ 
&+ \frac{4}{j^{1/4}} |\tilde z_1|^2 \mathrm{Re}(\tilde z_1)+ |\tilde z_1|^4+\frac{20}{j^{1/2}}\mathrm{Re}(\tilde{z}_2^3) + \frac{60}{j^{1/2}}\mathrm{Re}(\tilde{z}_2)|\tilde{z}_2|^2+\cdots,
\end{split}
\end{equation*}
where the dots denote the higher-order terms.

To define an anisotropic dilation,  let us denote by
 $\tau_{1j}:=\tau_1(\eta_j)=\frac{1}{2 j^{3/4}}, \; \tau_{2j}:=\tau_2(\eta_j)=\frac{1}{ \sqrt{15} j^{2/3}}$ for all $j\in \mathbb N_{\geq 1}$. Now let us introduce a sequence of  polynomial automorphisms $\phi_{{\eta}_j}$ of $\mathbb C^n$ ($j\in \mathbb N_{\geq 1}$), given by
\begin{equation*}
\begin{split}
&\phi_{{\eta}_j} ^{-1}(\tilde z_1,\tilde z_2,\tilde w)\\
&= \Big (\dfrac{1}{j^{1/4}}+\tau_{1j} \tilde z_1, \,\dfrac{1}{j^{1/6}}+ \tau_{2j} \tilde z_2, \,-\frac{2}{j}+\epsilon_j \tilde w+\frac{4}{j^{3/4}}\tau_{1j} \tilde z_1+\frac{2}{j^{1/2}}(\tau_{1j})^2 \tilde z_1^2+ \frac{6}{j^{5/6}} \tau_{2j}\tilde{z}_2+ \frac{6}{j^{2/3}}( \tau_{2j})^2 \tilde{z}_2^2)\Big).
\end{split}
\end{equation*}
Therefore, since $\tau_{1j}=o(1/j^{1/4})$ and $\tau_{2j}=o(1/j^{1/6})$ it follows that, for each $j\in\mathbb N_{\geq 1}$ the hypersurface $\phi_{\eta_j}(\{\rho=0\}) $ is then defined by

\begin{equation*}
\begin{split}
&\epsilon_j^{-1}\rho\circ \phi_{{\eta}_j'} ^{-1}(\tilde z_1,\tilde z_2,\tilde w)=\mathrm{Re}(\tilde w) +|\tilde z_1|^2+|\tilde z_2|^2+O(\frac{1}{j^{1/2}})=0.
\end{split}
\end{equation*} 
Hence, the sequence of domains $\Omega_j:=\phi_{{\eta}_j}(\mathcal {E}_{1,2,3}) $ converges normally to the following model
$$
\mathcal{D}_{1,1}:=\left \{(\tilde z_1,\tilde  z_2, \tilde  w)\in \mathbb C^3\colon \mathrm{Re}(\tilde  w)+|\tilde  z_1|^2 +|\tilde z_2|^2<0\right\},
$$
which is biholomorphically equivalent to the unit ball $\mathbb B^3$ in $\mathbb C^3$.

Now, we note that $\left|\eta_{j1}\dfrac{\partial P}{\partial z_1}(\eta_{j1}, \eta_{j2})\right|=2|\eta_{j1}|^4=\dfrac{2}{j}$ and $\left|\eta_{j2}\dfrac{\partial P}{\partial z_2}(\eta_{j1}, \eta_{j2})\right|=3|\eta_{j2}|^6=\dfrac{3}{j}$. Hence, the sequence $\{\eta_j=(\alpha_j,\beta_j)\}\subset  \Omega$ satisfies the  $(B,\xi_0)$-condition, and hence we have
$$
 \ell_{j1}\approx \ell_{j2}\approx \dfrac{|\eta_{j1}|^4}{\epsilon_j}=\dfrac{1/j}{1/j^2}=j\to +\infty
$$ 
 as $j\to \infty$. Therefore, we conclude that 
 \begin{align*}
d^2_{\mathcal{E}_{1,2,3}}(\eta_j;\xi)&\approx \frac{|\xi_3|^2}{\epsilon_j^2}+j \Big( \frac{|\xi_1|^2}{\tau_{j1}^2}+\frac{|\xi_2|^2}{\tau_{j2}^2}\Big)\approx \frac{|\xi_3|^2}{d_{\mathcal{E}_{1,2,3}}(\eta_j)^2}+\frac{|\xi_1|^2}{d_{\mathcal{E}_{1,2,3}}(\eta_j)^{5/4}}+\frac{|\xi_2|^2}{d_{\mathcal{E}_{1,2,3}}(\eta_j)^{7/6}};\\
K_{\mathcal{E}_{1,2,3}}(\eta_j, \eta_j) &\sim \frac{1}{4\pi^3 (\tau_{j1}  \tau_{j2})^2 \epsilon_j^2}\approx  \dfrac{1}{\left(d_{\mathcal{E}_{1,2,3}}(\eta_j)\right)^{2+3/4+2/3}};\\
\lim_{j\to\infty} \mathrm{Sec}_{\mathcal{E}_{1,2,3}}(\eta_j;\xi) &= -1;\lim_{j\to\infty} \mathrm{Ric}_{\mathcal{E}_{1,2,3}}(\eta_j;\xi) =-1; \lim_{j\to\infty} \mathrm{Scal}_{\mathcal{E}_{1,2,3}}(\eta_j;\xi)=-3.
\end{align*}

  \hfill $\Box$
\end{example}

\section{The boundary behavior of the Bergman kernel, the Bergman metric, and curvatures near a weakly pseudoconvex boundary point in $\mathbb C^2$} \label{S4}
\subsection{The spherically tangential convergence}\label{sub-sphere}

Let $\Omega$ be a domain in $\mathbb{C}^2$ with $\xi_0\in \partial \Omega$. We assume that $\partial\Omega$ is $\mathcal{C}^\infty$-smooth and pseudoconvex of finite D'Angelo type near $\xi_0$. By choosing appropriate coordinates $(z,w)$, we may assume that $\xi_0=(0,0)$ and the local defining function $\rho(z,w)$ for $\Omega$ near $\xi_0$ has the expansion
\begin{equation}\label{def-0}
\rho(z, w) =\mathrm{Re}(w) + H(z) +v\varphi(v, z)+ O(|z|^{2m+1}), 
\end{equation}
where $H$ is a real homogeneous subharmonic polynomial of degree $2m$ without harmonic terms, $2m$ is the D'Angelo type of $\partial \Omega$ at $\xi_0$, and $\varphi$ is a $\mathcal{C}^\infty$ function near the origin in $\mathbb{R}\times \mathbb{C}$ with $\varphi(0,0)=0$. The pseudoconvexity of $\partial \Omega$ ensures that $H$ is subharmonic and the type $2m$ is even.

Instead of strong $h$-extendibility, we need the following definition.
\begin{define}[See Definition $4.1$ in \cite{NNN25}]\label{spherically-convergence}
We say that a sequence $\{\eta_j=(\alpha_j,\beta_j)\}\subset  \Omega$ \emph{converges spherically $\frac{1}{2m}$-tangentially to $\xi_0$} if
\begin{itemize}
\item[(a)] $|\mathrm{Im}(\beta_j)|\lesssim |d_\Omega(\eta_j)|$;
\item[(b)] $|d_\Omega(\eta_j)|=o(|\alpha_{j}|^{2m})$;
\item[(c)]  $\Delta H(\alpha_{j})\gtrsim |\alpha_{j}|^{2m-2}$.
\end{itemize}
\end{define}

\begin{remark} For a smooth pseudoconvex domain $\Omega$ in $\mathbb{C}^2$, the condition $\mathrm{(c)}$ simply means that $\Omega$ is strongly pseudoconvex at the boundary points $\eta_j':=(\alpha_j,\beta_j+\epsilon_j)$ for all $j\in\mathbb{N}_{\geq 1}$, where $\{\epsilon_j\}\subset \mathbb{R}^+$ ensures that $\eta_j'\in \partial \Omega$.\end{remark}

\subsection{Estimates of Bergman kernel function and associated invariants near a weakly pseudoconvex boundary point in $\mathbb C^2$}

This subsection is devoted to the proofs of Theorem \ref{holo-sec-curvature-two} and Corollary \ref{cor:higher-dim-bergman-2}. Additionally, two typical examples are presented.

\begin{proof}[Proof of Theorem \ref{holo-sec-curvature-two}]
Let $\Omega$ and $\xi_0\in \partial \Omega $ be as in the statement of Theorem \ref{holo-sec-curvature-two}. As in Subsection \ref{sub-sphere}, we can choose coordinates $(z,w)$ such that $\xi_0=(0,0)$ and the defining function $\rho(z,w)$ has the expansion
\begin{equation}\label{def-0}
\rho(z, w) =\mathrm{Re}(w) + H(z) +v\varphi(v, z)+ O(|z|^{2m+1}), 
\end{equation}
where $H$ is a real homogeneous subharmonic polynomial of degree $2m$ without harmonic terms and $\varphi$ is a $\mathcal{C}^\infty$ function near the origin in $\mathbb{R}\times \mathbb{C}$ with $\varphi(0,0)=0$.

By the hypothesis of Theorem \ref{holo-sec-curvature-two}, let $\{\eta_j\}\subset\Omega$ be a sequence converging spherically $\frac{1}{2m}$-tangentially to $\xi_0$. We write $\eta_j=(\alpha_j,\beta_j)=(\alpha_{j},a_j+ib_j)$ for all $j\in \mathbb{N}_{\geq 1}$. Without loss of generality, we may assume that $\{\eta_j\}\subset U_0\cap \Omega:=U_0\cap\{\rho<0\}$. For each $j$, we consider the associated boundary point $\eta_j'=(\alpha_{j}, a_j +\epsilon_j+i b_j)\in\partial \Omega$, where $\{\epsilon_j\}\subset \mathbb{R}^+$ is appropriately chosen. We then have
\begin{itemize}
\item[(a)] $|b_j|\lesssim \epsilon_j$;
\item[(b)] $\epsilon_j=o(|\alpha_{j}|^{2m})$;
\item[(c)] $\Delta H(\alpha_{j})\gtrsim |\alpha_{j}|^{2m-2}$.
\end{itemize}

According to \cite[Section $3$]{Be03} and \cite[Proposition $1.1$]{Cat89}, for each point $\eta_j'$, there exists a biholomorphism $\Phi_{\eta_j'}$ of $\mathbb{C}^2$ with inverse $(z,w)=\Phi^{-1}_{\eta_j'}(\tilde{z}, \tilde{w})$ given by
$$
\Phi_{\eta_j'}^{-1}(z, w) =\left(\alpha_j + z, a_j+\epsilon_j+i b_j + d_0(\eta_j')w + \sum_{1 \leq k \leq 2m} d_k(\eta_j')z^k\right),
$$
where $d_0,\ldots,d_{2m}$ are $\mathcal{C}^\infty$ functions defined in a neighborhood of the origin in $\mathbb{C}^{2}$ with $d_0(0,0)=1$ and $d_1(0,0)=\cdots=d_{2m}(0,0)=0$, such that
\begin{equation}\label{Eq50}
\rho \circ \Phi_{\eta_j'}^{-1}(z, w) = \mathrm{Re}(w) + \sum_{\substack{j+k \leq 2m \\ j,k > 0}} a_{j,k}(\eta_j')z^j\bar{z}^k + O(|z|^{2m+1} + |z||w|). 
\end{equation}

We first define
\begin{equation*}
A_l(\eta_j')=\max \left\{|a_{j,k}(\eta_j')| : j+k=l\right\} \quad (2\leq l \leq 2m).
\end{equation*}
Then we define $\tau(\eta_j',\epsilon_j)$ by 
\begin{equation*}
\tau_j=\tau(\eta_j',\epsilon_j)=\min \left\{\left( \frac{\epsilon_j}{A_l(\eta_j')} \right)^{1/l} : 2\leq l \leq 2m\right \}.
\end{equation*}
Since the type of $\partial \Omega$ at $\xi_0$ equals $2m$, we have $A_{2m}(\xi_0)\neq 0$. Thus, if $U_0$ is sufficiently small, then $|A_{2m}(\eta_j')|\geq c>0$ for all $\eta_j'\in U_0$. This yields the estimate
\begin{equation*}
\epsilon_j^{1/2m}\lesssim \tau(\eta_j',\epsilon_j) \lesssim \epsilon_j^{1/2} \quad (\eta_j'\in U_0).
\end{equation*}

To complete the scaling procedure, we define the anisotropic dilation $\Delta_j$ by 
\[
\Delta_j (z,w)=\left( \frac{z}{\tau_{j}},\frac{w}{\epsilon_j}\right), \quad j\in \mathbb{N}_{\geq 1}.
\]
As in the proof of Theorem \ref{holo-sec-curvature}, we have $\Delta_j\circ \Phi_{\eta_j'}(\eta_j')=(0,0)$ and $\Delta_j\circ \Phi_{\eta_j'}(\eta_j)=(0,-1/d_0(\eta_j'))\to(0,-1)$ as $j\to\infty$, since $d_0(\eta_j')\to 1$ as $j\to\infty$. In addition, let us define $\rho_j(z,w):=\epsilon_j^{-1}\rho\circ \Phi_{\eta_j'}^{-1}\circ(\Delta_j)^{-1}(z,w)$ for $j\in \mathbb{N}_{\geq 1}$. Then \eqref{Eq50} yields that
\begin{equation*}
\rho_j(z,w)=\mathrm{Re}(w)+ P_{\eta_j'}(z)+O(\tau(\eta_j',\epsilon_j)),
\end{equation*}
where
\begin{equation*}
P_{\eta_j'}(z):=\sum_{\substack{k+l\leq 2m\\
k,l>0}} a_{k,l}(\eta_j') \epsilon_j^{-1} \tau_j^{k+l}z^k \bar{z}^l.
\end{equation*}

Next, we write $H(z)=\sum_{j=1}^{2m-1} a_j z^j\bar{z}^{2m-j}$ and set $z=|z| e^{i\theta}$. This gives $H(z)=|z|^{2m} g(\theta)$ for some function $g(\theta)$. Following the approach in \cite{BF78a}, the Laplacian of $H$ satisfies
$$
\Delta H(z)=|z|^{2m-2} \left((2m)^{2} g(\theta)+g_{\theta\theta}(\theta)\right)\geq 0.
$$
By \cite[Lemma $4.1$]{NNN25}, we also have$$
\frac{\partial^2 H(\alpha_j)}{\partial z\partial \bar z}\epsilon_j^{-1}\tau_j^{2}=(2m)^2g(\theta_j)+g_{\theta\theta}(\theta_j),\; \forall j\geq 1,
$$
where $\alpha_{j}=|\alpha_{j}|e^{\theta_j}$, $j\geq 1$. Because of the condition (c), without loss of generality we may assume that the limit $\displaystyle a:= \lim_{j\to \infty} \dfrac{1}{2}\frac{\partial^2 H}{\partial z\partial \bar z}(\alpha_{j})\epsilon_j^{-1} \tau_{j}^2$ exists. 

Direct computation yields that
\begin{equation}\label{eq-1-1}
a_{l,k-l}(\eta_j')=\frac{1}{k!}\frac{\partial^{k} \rho}{\partial z^l\partial \bar{z}^{k-l}}(\eta_j')=\frac{1}{k!}\frac{\partial^{k} H}{\partial z^l\partial \bar{z}^{k-l}}(\alpha_j)+ \frac{b_j}{k!}\frac{\partial^{k} \varphi}{\partial z^l\partial \bar{z}^{k-l}}(b_j,\alpha_j)+\cdots
\end{equation}
for $j\in \mathbb{N}_{\geq 1}$, $2\leq k\leq 2m$, and $0\leq l\leq k$, where the dots represent higher-order terms. 

Since $H$ is homogeneous of degree $2m$ and subharmonic, we have $\left|\frac{\partial^{k} H}{\partial z^l\partial \bar{z}^{k-l}}(\alpha_j)\right|\lesssim |\alpha_{j}|^{2m-k}$ for $2\leq k\leq 2m$. Using the estimate $|b_j|\lesssim \epsilon_j=o( |\alpha_{j}|^{2m})$, we obtain $|a_{l,k-l}(\eta_j')|\lesssim |\alpha_{j}|^{2m-k}$ for $2\leq k\leq 2m$. This gives $A_k(\eta_j')\lesssim |\alpha_{j}|^{2m-k}$, which leads to
$$
\left(\frac{\epsilon_j}{A_k(\eta_j')} \right)^{1/k}\gtrsim\left(\frac{\epsilon_j}{|\alpha_{j}|^{2m-k}}\right)^{1/k}=|\alpha_{j}| \left(\frac{\epsilon_j}{|\alpha_{j}|^{2m}}\right)^{1/k}, \quad 2\leq k\leq 2m.
$$
Moreover, since $\epsilon_j=o( |\alpha_{j}|^{2m})$ and $|\alpha_{j}| \big(\epsilon_j/ |\alpha_{j}|^{2m}\big)^{1/2}=o\left(|\alpha_{j}| \big(\epsilon_j/ |\alpha_{j}|^{2m}\big)^{1/k}\right)$ for all $k\geq3$, it follows that
$$
\tau_j=\left(\frac{\epsilon_j}{A_2(\eta_j')} \right)^{1/2}\approx|\alpha_{j}| \left(\frac{\epsilon_j}{|\alpha_{j}|^{2m}}\right)^{1/2}.
$$

We proceed to establish convergence for the sequence $\{\Delta_j\circ \Phi_{\eta_j'}(U_0\cap \Omega)\}_{j=1}^{\infty}$. A direct calculation shows that
\begin{align*}
| a_{l,k-l}(\eta_j')| \epsilon_j^{-1}\tau_j^{k}\approx\left |\frac{\partial^k H}{\partial z^l\partial \bar z^{k-l}}(\alpha_j)\right| \epsilon_j^{-1}\tau_j^{k}&\lesssim |\alpha_j|^{2m-k} \epsilon_j^{-1}\tau_j^{k}=|\alpha_j|^{2m} \epsilon_j^{-1}\Big(\frac{\tau_j}{|\alpha_j|}\Big)^{k}\\
&\lesssim \frac{|\alpha_j|^{2m}}{ \epsilon_j}\Big(\frac{\epsilon_j}{|\alpha_j|^{2m}}\Big)^{k/2}=\Big(\frac{\epsilon_j}{|\alpha_j|^{2m}}\Big)^{k/2-1}.
\end{align*}
This implies that $a_{l,k-l}(\eta_j') \epsilon_j^{-1}\tau_j^{k}\to 0$ as $j\to\infty$ for $3\leq k\leq 2m$ and 
$$
\lim_{j\to \infty}a_{1,1}(\eta_j')\epsilon_j^{-1} \tau_{j}^2 =\lim_{j\to \infty} \dfrac{1}{2}\frac{\partial^2 H}{\partial z\partial \bar z}(\alpha_{j})\epsilon_j^{-1} \tau_{j}^2=a >0.
$$
Altogether, after extracting a subsequence if necessary, the sequence $\{ \rho_j\}$ converges on compacta to the following function
$$
\hat\rho(z,w):=\mathrm{Re}(w)+a|z|^2,
$$
where $\displaystyle a=\frac{1}{2} \lim_{j\to \infty} \frac{\partial^2 H}{\partial z\partial \bar z}(\alpha_{j})\epsilon_j^{-1} \tau_{j}^2 > 0$. Therefore, by passing to a subsequence if necessary, we may assume that the sequences $\Omega_j:=\Delta_j\circ \Phi_{\eta_j'}(\Omega)$ and $\Delta_j\circ \Phi_{\eta_j'}(\Omega\cap U_0)$ converge normally to the Siegel half-space
$$
M_{a}:=\left \{( z,w)\in \mathbb{C}^2\colon \hat\rho(z,w)=\mathrm{Re}(w)+a|z|^2<0\right\}.
$$

Now we first define the linear transformation $\Theta$ by 
\[ 
\tilde{w}= w, \quad \tilde{z}=\sqrt{a}\, z,
\]
which maps $M_{a}$ onto the Siegel half-space
$$
\mathcal{U}_2:=\{(z,w)\in \mathbb{C}^2 \colon\mathrm{Re}(w) +|z|^2<0\}.
$$
Subsequently, the holomorphic map $\Psi$ defined by
\[
(z,w)\mapsto \left( \frac{2z}{1-w}, \frac{w+1}{1-w}\right)
\]
is a biholomorphism from $\mathcal{U}_2$ onto $\mathbb{B}^2$.

Next, let us consider the sequence of biholomorphic maps $f_j:=\Psi\circ \Theta\circ \Delta_j\circ \Phi_{\eta_j'} \colon \Omega \to f_j(\Omega)=\Psi\circ \Theta(\Omega_j)$. One notes that $f_j(\Omega \cap U_0)$ converges normally to $\mathbb{B}^{2}$ and $f_j(\partial \Omega \cap U_0)$ converges to $\partial \mathbb{B}^{2}$. Moreover, since $\Theta(0,-1) =(0,-1)$ and $\Psi(0,-1)=(0,0)$, it follows that
\[
f_j(\eta_j) = \Psi\circ \Theta(0,-1/d_0(\eta_j')) = \Psi(0,-1/d_0(\eta_j')) = \left(0, \frac{1-1/d_0(\eta_j')}{1+1/d_0(\eta_j')}\right)  \to (0, 0) \quad \text{as } j \to \infty.
\]
Therefore, by a similar argument as in the proof of Theorem \ref{holo-sec-curvature}, we conclude that for a sufficiently small $\epsilon > 0$, there exists $j_0 \in \mathbb{N}_{\geq 1}$ such that
\[
B\big((0,0), 1 - \epsilon\big) \subset F_j(U_0\cap \Omega) \subset B\big((0,0), 1 + \epsilon\big), \, \forall j \geq j_0,
\]
where $F_j(\cdot) := f_j(\cdot) - f_j(\eta_j)$ for all $j \geq j_0$.

In the sequel, we estimate the Bergman kernel function, Bergman metric, and associated curvatures of $\Omega$ at $\eta_j$ in the direction $\displaystyle\xi= \xi_1\dfrac{\partial}{\partial z}+\xi_{2}\dfrac{\partial}{\partial w}\in T^{1,0}_{\eta_j} \Omega\setminus \{0\}$. To do this, we compute the Jacobian matrices of the component mappings. Indeed, a computation shows that
\begin{align*}
d\Phi_{\eta_j'}\big|_{\eta_j} &= \begin{pmatrix}
1 & 0 \\
-\frac{d_1(\eta_j')}{d_0(\eta_j')} & \frac{1}{d_0(\eta_j')}
\end{pmatrix}, \quad \det (d\Phi_{\eta_j'}\big|_{\eta_j}) = \frac{1}{d_0(\eta_j')} \sim 1;\\
d\Psi\big|_{(0,-1/d_0(\eta_j'))} &= \begin{pmatrix}
\frac{2}{1+1/d_0(\eta_j')} & 0 \\
0 & \frac{2}{(1+1/d_0(\eta_j'))^2}
\end{pmatrix}\sim \begin{pmatrix}
1 & 0 \\
0 & \frac{1}{2}
\end{pmatrix}, \det (d\Psi\big|_{(0,-1/d_0(\eta_j'))}) \sim\frac{1}{2}.
\end{align*}
In addition, since the maps $\Theta$ and $\Delta_j$ are linear, we conclude that
\begin{align*}
dF_j(\xi) &= d\Psi\big|_{(0,-1/d_0(\eta_j'))} \circ d\Theta \circ d\Delta_j \circ d\Phi_{\eta_j'}\big|_{\eta_j}(\xi)\\
&= \begin{pmatrix}
\frac{2}{1+1/d_0(\eta_j')} & 0 \\
0 & \frac{2}{(1+1/d_0(\eta_j'))^2}
\end{pmatrix}\begin{pmatrix}
\sqrt{a} & 0 \\
0 & 1
\end{pmatrix} \begin{pmatrix}
\frac{1}{\tau_j} & 0 \\
0 & \frac{1}{\epsilon_j}
\end{pmatrix} \begin{pmatrix}
1 & 0 \\
-\frac{d_1(\eta_j')}{d_0(\eta_j')} & \frac{1}{d_0(\eta_j')}
\end{pmatrix} \begin{pmatrix}
\xi_1 \\
\xi_2
\end{pmatrix}\\
&\sim \begin{pmatrix}
1 & 0 \\
0 & \frac{1}{2}
\end{pmatrix}\begin{pmatrix}
\sqrt{a} & 0 \\
0 & 1
\end{pmatrix} \begin{pmatrix}
\frac{1}{\tau_j} & 0 \\
0 & \frac{1}{\epsilon_j}
\end{pmatrix}
 \begin{pmatrix}
\xi_1 \\
\frac{\xi_2}{d_0(\eta_j')}-\frac{d_1(\eta_j')\xi_1}{d_0(\eta_j')}
\end{pmatrix}\\
&= \begin{pmatrix}
\frac{\sqrt{a}}{\tau_j} & 0 \\
0 & \frac{1}{2\epsilon_j }
\end{pmatrix} 
 \begin{pmatrix}
\xi_1 \\
\frac{\xi_2}{d_0(\eta_j')}-\frac{d_1(\eta_j')\xi_1}{d_0(\eta_j')}
\end{pmatrix}\\
&= \left(\frac{\sqrt{a}\xi_1}{\tau_j}, \frac{\xi_2-d_1(\eta_j')\xi_1}{2\epsilon_jd_0(\eta_j')}\right)
\end{align*}
for $\xi = (\xi_1, \xi_2) \in T_{\eta_j}^{1,0}\Omega$. 

We shall estimate the coefficients $d_0(\eta_j'), d_1(\eta_j')$. Indeed, following the proof of Theorem~\ref{holo-sec-curvature} we conclude that 
\begin{align*}
\dfrac{1}{d_0(\eta_j')}=2\frac{\partial \rho}{\partial w}(\eta_j')\sim 1\text{ and }-\dfrac{d_1(\eta_j')}{d_0(\eta_j')}=2\frac{\partial \rho}{\partial z}(\eta_j').
\end{align*}
Let us denote by 
$$
\ell_j:=\epsilon_j^{-1}\left| \frac{\partial \rho}{\partial z}(\eta_j') \right|\tau_j,\; j\geq 1.
$$
Since $B\big((0,0), 1 - \epsilon\big) \subset F_j(U_0\cap \Omega) \subset B\big((0,0), 1 + \epsilon\big)$ for all $j$ large enough and $F_j(\eta_j) = (0, 0)$, by Lemma \ref{local-approx} and Corollary ~\ref{cor-approx-holo-cur} it follows that
\begin{equation}\label{Bergmametric-eq2}
\begin{split}
d^2_\Omega(\eta_j;\xi)\sim d^2_{U_0\cap \Omega}(\eta_j;\xi)&\sim ds_{\mathbb{B}^{2}}^2(0; dF_j(\xi), dF_j(\xi)) = 4|dF_j(\xi)|^2\\
&\sim 4\left[\frac{a|\xi_1|^2}{\tau_j^2} + \frac{|\xi_2-d_1(\eta_j')\xi_1|^2}{4\epsilon_j^2}\right]\\
&\approx   \frac{|\xi_2|^2}{\epsilon_j^2}+\max\{\ell_j,1\}\frac{|\xi_1|^2}{\tau_j^2}.
\end{split}
\end{equation}
Next,  the transformation rule for the Bergman kernel function implies that
$$
K_{U_0\cap \Omega}(\eta_j, \eta_j) = K_{F_j(U_0\cap \Omega)}(F_j(\eta_j), F_j(\eta_j)) |J_{F_j}(\eta_j)|^2.
$$
The holomorphic Jacobian determinant is given by
\begin{align*}
\det J_{\mathbb{C}}(F_j) &= \det(d\Psi\big|_{(0,-1/d_0(\eta_j'))}) \cdot \det(d\Theta) \cdot \det(d\Delta_j) \cdot \det(d\Phi_{\eta_j'}\big|_{\eta_j})\\
&= \frac{4}{(1+1/d_0(\eta_j'))^3} \cdot \sqrt{a} \cdot \frac{1}{\tau_j \epsilon_j} \cdot \frac{1}{d_0(\eta_j')}\\
&\sim \frac{\sqrt{a}}{2 \tau_j \epsilon_j}.
\end{align*}
As $F_j(\eta_j) = 0 = (0, 0)$ and $B\big((0,0), 1 - \epsilon\big) \subset F_j(U_0\cap \Omega) \subset B\big((0,0), 1 + \epsilon\big)$ for all $j$ large enough, by Lemma \ref{local-approx} and Corollary ~\ref{cor-approx-holo-cur} one obtains
\begin{align*}
K_{\Omega}(\eta_j, \eta_j)\sim K_{U_0\cap \Omega}(\eta_j, \eta_j) &\sim K_{\mathbb{B}^{2}}(0, 0) |\det J_{\mathbb{C}}(F_j)|^2 = \frac{1}{\pi^{2}} |\det J_{\mathbb{C}}(F_j)|^2\\
&\sim \frac{a}{4\pi^{2} \tau_j^2 \epsilon_j^2}\approx \frac{1}{ \tau_j^2 \epsilon_j^2}.
\end{align*}

Finally, by Corollaries ~\ref{cor-approx-holo-cur} and \ref{cor-approx-holo-cur-2}, we conclude that 

\begin{align*}
\lim_{j\to\infty} \mathrm{Sec}_\Omega(\eta_j;\xi) &= \lim_{j\to\infty} \mathrm{Sec}_{F_j(\Omega)}(F_j(\eta_j);dF_j(\eta_j)(\xi)) \\
&= \lim_{j\to\infty} \mathrm{Sec}_{F_j(\Omega)}\left((0,0);\frac{dF_j(\eta_j)(\xi)}{|dF_j(\eta_j)(\xi)|}\right) = -\frac{4}{3},
\end{align*}
for any $\xi = (\xi_1, \xi_2) \in T_{\eta_j}^{1,0}\Omega\setminus \{0\}$. Similarly, we also obtain
$$
\lim_{j\to\infty} \mathrm{Ric}_{\Omega}(\eta_j;\xi) =-1, \quad \lim_{j\to\infty} \mathrm{Scal}_{\Omega}(\eta_j;\xi)=-2.
$$
This completes the proof of Theorem~\ref{holo-sec-curvature-two}.
\end{proof}

\begin{proof}[Proof of Corollary \ref{cor:higher-dim-bergman-2}] By our assumption, we have
$$
\left|\alpha_{j}\frac{\partial H(\alpha_j)}{\partial z}\right|\approx  |\alpha_{j}|^{2m}.
$$
Since $|b_j|\lesssim \epsilon_j=o( |\alpha_{j}|^{2m})$, arguing similarly to \eqref{eq-1-1}, we obtain 
$$\left|\alpha_j\frac{\partial \rho}{\partial z}(\eta_j)\right|\sim \left|\alpha_j \frac{\partial H}{\partial z}(\alpha_j)\right| \approx|\alpha_{j}|^{2m}.$$
Therefore, one has
$$
\ell_j:=\left(\epsilon_j^{-1}\tau_j\left| \frac{\partial \rho}{\partial z}(\eta_j) \right|\right)^2\approx \left(\epsilon_j^{-1}\left| \alpha_{j}\frac{\partial \rho}{\partial z}(\eta_j) \right| \left(\frac{\epsilon_j}{|\alpha_{j}|^{2m}}\right)^{1/2}\right)^2  \approx \frac{|\alpha_{j}|^{2m}}{\epsilon_j}  \to +\infty
$$
as $j \to \infty$. Consequently, \eqref{Bergmametric-eq2} becomes
$$
d^2_\Omega(\eta_j;\xi)\approx \frac{|\xi_2|^2}{\epsilon_j^2}+\ell_j   \frac{|\xi_1|^2}{\tau_{j}^2},
$$
as desired.\end{proof}

We close this subsection with two examples. First of all, the following example illustrates spherically $\frac{1}{2m}$-tangential convergence.
\begin{example}\label{Kohn-Nirenberg}
Let $\Omega_{KN}$ be the Kohn-Nirenberg domain in $\mathbb{C}^2$, that does not admit a holomorphic support function (see \cite{KN73}) and is recently demonstrated uniformly squeezing in \cite{FRW25}, defined by
$$
\Omega_{KN}:=\left\{(z,w)\in \mathbb{C}^2\colon \mathrm{Re}(w)+ |z|^8+\frac{15}{7}|z|^2\mathrm{Re}(z^6)<0\right\}.
$$
Let us consider a bounded domain $D$ with $(0,0)\in \partial \Omega$ such that $D\cap U_0=\Omega_{KN}\cap U_0$ for some neighbourhood $U_0$ of $(0,0)$ in $\mathbb{C}^2$. We denote by $\rho(z,w)=\mathrm{Re}(w)+ |z|^8+\frac{15}{7}|z|^2\mathrm{Re}(z^6)$ and $P(z)=|z|^8+\frac{15}{7}|z|^2\mathrm{Re}(z^6)$. It is easy to see that $\Delta P(z)=4(16|z|^6+15\mathrm{Re}(z^6))\geq 4|z|^6$, and hence $\partial \Omega$ is strongly $h$-extendible at $(0,0)$.

We first consider a sequence $\eta_j=\Big(\frac{1}{j^{1/8}},-\frac{22}{7j}-\frac{1}{j^2}\Big)\in D$ for every $j\in \mathbb{N}_{\geq 1}$. Then the sequence $\left\{\left(\frac{1}{j^{1/8}},-\frac{22}{7j}-\frac{1}{j^2}\right)\right\}$ converges spherically $\frac{1}{8}$-tangentially to $(0,0)$. Moreover, we have $\rho(\eta_j)=-\frac{22}{7j}-\frac{1}{j^2}+\frac{22}{7j}=-\frac{1}{j^2}\approx -d_{\Omega_{KN}}(\eta_j)$. Setting $\epsilon_j=|\rho(\eta_j)|=\frac{1}{j^2}$ and substituting $\xi = z - \frac{1}{j^{1/8}}$ to  the formulas
\begin{align*}
|\xi + \frac{1}{j^{1/8}}|^8 &= \frac{1}{j} + \frac{8}{j^{7/8}}\mathrm{Re}(\xi) + \frac{16}{j^{3/4}}|\xi|^2 + \frac{12}{j^{3/4}}\mathrm{Re}(\xi^2) + O\Big(\frac{1}{j^{5/8}}\Big);\\
|\xi + \frac{1}{j^{1/8}}|^2\mathrm{Re}\Big((\xi + \frac{1}{j^{1/8}})^6\Big)&= \frac{1}{j} + \frac{8}{j^{7/8}}\mathrm{Re}(\xi) + \frac{7}{j^{3/4}}|\xi|^2 + \frac{21}{j^{3/4}}\mathrm{Re}(\xi^2) + O\Big(\frac{1}{j^{5/8}}\Big),
\end{align*}
we obtain that

\begin{align*}
&\rho(z,w)\\
&=\mathrm{Re}(w)+ \Big|\Big(z-\frac{1}{j^{1/8}}\Big)+\frac{1}{j^{1/8}}\Big|^8+\frac{15}{7}\Big|\Big(z-\frac{1}{j^{1/8}}\Big)+\frac{1}{j^{1/8}}\Big|^2\mathrm{Re}\Big(\Big(\Big(z-\frac{1}{j^{1/8}}\Big)+\frac{1}{j^{1/8}}\Big)^6\Big)\\
&=\mathrm{Re}(w)+\frac{1}{j}+\frac{8}{j^{7/8}} \mathrm{Re}\Big(z-\frac{1}{j^{1/8}}\Big)+\frac{16}{j^{3/4}} \Big|z-\frac{1}{j^{1/8}}\Big|^2+\frac{12}{j^{3/4}} \mathrm{Re}\Big(\Big(z-\frac{1}{j^{1/8}}\Big)^2\Big)\\
&\quad+\frac{15}{7}\left[\frac{1}{j}+\frac{8}{j^{7/8}}\mathrm{Re}\Big(z-\frac{1}{j^{1/8}}\Big)+\frac{21}{j^{3/4}} \mathrm{Re}\Big(\Big(z-\frac{1}{j^{1/8}}\Big)^2\Big)+\frac{7}{j^{3/4}}\Big|z-\frac{1}{j^{1/8}}\Big|^2\right]+\cdots\\
&=\mathrm{Re}(w)+\frac{22}{7j}+\frac{176}{7j^{7/8}} \mathrm{Re}\Big(z-\frac{1}{j^{1/8}}\Big)+\frac{57}{j^{3/4}} \mathrm{Re}\Big(\Big(z-\frac{1}{j^{1/8}}\Big)^2\Big)+\frac{31}{j^{3/4}}\Big|z-\frac{1}{j^{1/8}}\Big|^2\\
&\quad +O\Big(\frac{1}{j^{5/8}}\Big|z-\frac{1}{j^{1/8}}\Big|^3\Big).
\end{align*}

To define an anisotropic dilation, let us denote $\tau_j:=\tau(\eta_j)=\frac{1}{j^{5/8}}$ for all $j\in \mathbb{N}_{\geq 1}$. Now we introduce a sequence of polynomial automorphisms $\phi_{{\eta}_j}^{-1}$ of $\mathbb{C}^2$, given by
\begin{align*}
 z&=\frac{1}{j^{1/8}}+\tau_j \tilde{z};\\
w&=\epsilon_j \tilde{w}-\frac{22}{7j}-\frac{176}{7j^{7/8}} \tau_j \tilde{z}-\frac{57}{j^{3/4}}  \tau_j^2 \tilde{z}^2.
\end{align*}
Therefore, since $\tau_j=\frac{1}{j^{5/8}}=o\big(\frac{1}{j^{1/8}}\big)$, we have
\begin{equation*}
\epsilon_j^{-1}\rho\circ \phi_{{\eta}_j} ^{-1}(\tilde{z},\tilde{w})= \mathrm{Re}(\tilde{w}) + 31|\tilde{z}|^2 +O\Big(\frac{1}{j^{1/2}}\Big).
\end{equation*} 
This implies that $D_j:=\phi_{{\eta}_j}(D)$ converges normally to the model $\mathcal{H}:=\{(\tilde{z},\tilde{w})\in \mathbb{C}^2\colon \mathrm{Re}(\tilde{w}) + 31|\tilde{z}|^2<0 \}$, which is biholomorphically equivalent to $\mathbb{B}^2$, and $\phi_{{\eta}_j}(\eta_j)=(0,-1)\in \mathcal{H}$ for all $j\geq 1$.

A computation shows that the Jacobian matrix of $\phi_{\eta_j}^{-1}$ is given by
\[
 d\phi_{\eta_j}^{-1}(\tilde z,\tilde w)= \begin{pmatrix}
\tau_j & 0 \\
-\frac{176}{7j^{7/8}} \tau_j - \frac{114}{j^{3/4}} \tau_j^2 \tilde{z} & \epsilon_j
\end{pmatrix}.
\]
Therefore, the Jacobian matrix of $\phi_{\eta_j}$ is given by
\[
d\phi_{\eta_j}(z,w) = \begin{pmatrix}
\frac{1}{\tau_j} & 0 \\
\frac{1}{\epsilon_j}\left(\frac{176}{7j^{7/8}} + \frac{114}{j^{3/4}} \tau_j \tilde{z}\right) & \frac{1}{\epsilon_j}
\end{pmatrix}
\]
 Hence, we get
\[
d\phi_{\eta_j}(\eta_j) = \begin{pmatrix}
\frac{1}{\tau_j} & 0 \\
\frac{176}{7j^{7/8}\epsilon_j} & \frac{1}{\epsilon_j}
\end{pmatrix}.
\]
Note that $\frac{\partial \rho}{\partial z}(\eta_j)\ne 0$ and following the proof of Theorem ~\ref{holo-sec-curvature-two}, we obtain
\begin{align*}
d^2_D(\eta_j;\xi)\approx \frac{|\xi_2|^2}{\epsilon_j^2}+\ell_j\frac{|\xi_1|^2}{\tau_{j}^2}\approx \frac{|\xi_2|^2}{\epsilon_j^2}+j\;\frac{|\xi_1|^2}{\tau_{j}^2},
\end{align*}
where 
$$
\ell_j:=\left(\epsilon_j^{-1}\tau_j\left| \frac{\partial \rho}{\partial z}(\eta_j) \right|\right)^2=\left(\epsilon_j^{-1}\tau_j  \frac{176}{7j^{7/8}}  \right)^2\approx j.
$$
In addition, we have
\begin{align*}
K_{D}(\eta_j, \eta_j) \sim  \frac{a}{4\pi^{2} \tau_j^2 \epsilon_j^2}= \frac{31}{4\pi^{2} \tau_j^2 \epsilon_j^2}.
\end{align*}

\end{example}

Finally, the following example demonstrates the case that $\{\eta_j\}$ does not satisfy the $(B,\xi_0)$-condition.

\begin{example}\label{non-balance}
Let $\widetilde \Omega_{KN}$ be the modified Kohn-Nirenberg domain in $\mathbb{C}^2$ given by
$$
\widetilde\Omega_{KN}:=\left\{(z,w)\in \mathbb{C}^2\colon \mathrm{Re}(w)+ |z|^8-|z|^2\mathrm{Re}(z^6)<0\right\}.
$$
Let us consider a bounded domain $\Omega$ with $(0,0)\in \partial \Omega$ such that $\Omega\cap U_0=\widetilde\Omega_{KN}\cap U_0$ for some neighbourhood $U_0$ of $(0,0)$ in $\mathbb{C}^2$. We denote by $\rho(z,w)=\mathrm{Re}(w)+ |z|^8-|z|^2\mathrm{Re}(z^6)$ and $P(z)=|z|^8-|z|^2\mathrm{Re}(z^6)$. It is easy to see that $\Delta P(z)=4(16|z|^6-7\mathrm{Re}(z^6))\geq 36|z|^6$, and hence $\partial \Omega$ is strongly $h$-extendible at $(0,0)$. 

We first consider a sequence $\eta_j=\Big(\frac{1}{j^{1/8}},-\frac{1}{j^2}\Big)\in \Omega$ for every $j\in \mathbb{N}_{\geq 1}$. Then the sequence $\left\{\left(\frac{1}{j^{1/8}},-\frac{1}{j^2}\right)\right\}$ converges spherically $\frac{1}{8}$-tangentially to $(0,0)$. Moreover,  $\rho(\eta_j) = -\frac{1}{j^2} + 0 = -\frac{1}{j^2}$ and hence then sequence $\eta_j'=\Big(\frac{1}{j^{1/8}},0\Big)\in \partial \Omega$ for every $j\in \mathbb{N}_{\geq 1}$. Setting $\epsilon_j=|\rho(\eta_j)|=\frac{1}{j^2}$ and by argument as in Example \ref{Kohn-Nirenberg}, one gets
\begin{align*}
&\rho(z,w)\\
&=\mathrm{Re}(w)+ \Big|\Big(z-\frac{1}{j^{1/8}}\Big)+\frac{1}{j^{1/8}}\Big|^8-\Big|\Big(z-\frac{1}{j^{1/8}}\Big)+\frac{1}{j^{1/8}}\Big|^2\mathrm{Re}\Big(\Big(\Big(z-\frac{1}{j^{1/8}}\Big)+\frac{1}{j^{1/8}}\Big)^6\Big)\\
&= \mathrm{Re}(w) + \frac{1}{j} + \frac{8}{j^{7/8}}\mathrm{Re}\Big(z-\frac{1}{j^{1/8}}\Big) + \frac{16}{j^{3/4}}\Big|z-\frac{1}{j^{1/8}}\Big|^2\\
&\quad + \frac{12}{j^{3/4}}\mathrm{Re}\Big(\Big(z-\frac{1}{j^{1/8}}\Big)^2\Big) - \frac{1}{j} - \frac{8}{j^{7/8}}\mathrm{Re}\Big(z-\frac{1}{j^{1/8}}\Big)\\
&\quad - \frac{7}{j^{3/4}}\Big|z-\frac{1}{j^{1/8}}\Big|^2 - \frac{21}{j^{3/4}}\mathrm{Re}\Big(\Big(z-\frac{1}{j^{1/8}}\Big)^2\Big) + O\Big(\frac{1}{j^{5/8}}\Big|z-\frac{1}{j^{1/8}}\Big|^3\Big)\\
&= \mathrm{Re}(w) + \frac{9}{j^{3/4}}\Big|z-\frac{1}{j^{1/8}}\Big|^2 - \frac{9}{j^{3/4}}\mathrm{Re}\Big(\Big(z-\frac{1}{j^{1/8}}\Big)^2\Big)+ O\Big(\frac{1}{j^{5/8}}\Big|z-\frac{1}{j^{1/8}}\Big|^3\Big).
\end{align*}

To define an anisotropic dilation, let us denote $\tau_j:=\tau(\eta_j)=\frac{1}{j^{5/8}}$ for all $j\in \mathbb{N}_{\geq 1}$. Then we introduce a sequence of polynomial automorphisms $\phi_{{\eta}_j}^{-1}$ of $\mathbb{C}^2$, given by
\begin{align*}
 z&=\frac{1}{j^{1/8}}+\tau_j \tilde{z};\\
w&=\epsilon_j \tilde{w}-\frac{9}{j^{3/4}}  \tau_j^2 \tilde{z}^2.
\end{align*}
Therefore, since $\tau_j=\frac{1}{j^{5/8}}=o\big(\frac{1}{j^{1/8}}\big)$ and $\epsilon_j = \frac{1}{j^2}$, we have
\begin{equation*}
\epsilon_j^{-1}\rho\circ \phi_{{\eta}_j} ^{-1}(\tilde{z},\tilde{w})= \mathrm{Re}(\tilde{w}) + 9|\tilde{z}|^2 +O\Big(\frac{1}{j^{1/8}}\Big).
\end{equation*} 
This implies that $\Omega_j:=\phi_{{\eta}_j}(\Omega)$ converges normally to the model $\mathcal{H}:=\{(\tilde{z},\tilde{w})\in \mathbb{C}^2\colon \mathrm{Re}(\tilde{w}) + 9|\tilde{z}|^2<0 \}$, which is biholomorphically equivalent to $\mathbb{B}^2$, and $\phi_{{\eta}_j}(\eta_j)=(0,-1)\in \mathcal{H}$ for all $j\geq 1$. 

A computation shows that the Jacobian matrix of $\phi_{\eta_j}^{-1}$ is given by
\[
d\phi_{\eta_j}^{-1}(\tilde z,\tilde w)= \begin{pmatrix}
\tau_j & 0 \\
-\frac{18}{j^{3/4}} \tau_j^2 \tilde{z} & \epsilon_j
\end{pmatrix}
\]
and, therefore the Jacobian matrix of $\phi_{\eta_j}$ is given by
\[
d\phi_{\eta_j}(z,w) = \begin{pmatrix}
\frac{1}{\tau_j} & 0 \\
\frac{18}{j^{3/4}} \frac{\tau_j}{\epsilon_j} \tilde{z} & \frac{1}{\epsilon_j}
\end{pmatrix}=
 \begin{pmatrix}
\frac{1}{\tau_j} & 0 \\
\frac{18}{j^{3/4}} \frac{1}{\epsilon_j} (z-\frac{1}{j^{1/8}}) & \frac{1}{\epsilon_j}
\end{pmatrix}.
\]
Hence, we obtain
\[
d\phi_{\eta_j}(\eta_j) = \begin{pmatrix}
\frac{1}{\tau_j} & 0 \\
0 & \frac{1}{\epsilon_j}
\end{pmatrix}.
\]
Note that $\frac{\partial \rho}{\partial z}(\eta_j)=0$ and  following the proof of Theorem ~\ref{holo-sec-curvature-two}, we get
\begin{align*}
d^2_\Omega(\eta_j;\xi)\approx \frac{|\xi_2|^2}{\epsilon_j^2}+\frac{|\xi_1|^2}{\tau_{j}^2}.
\end{align*}
In addition, we have
\begin{align*}
K_{\Omega}(\eta_j, \eta_j) \sim  \frac{a}{4\pi^{2} \tau_j^2 \epsilon_j^2}= \frac{9}{4\pi^{2} \tau_j^2 \epsilon_j^2}.
\end{align*}

\end{example}

\begin{acknowledgement} The author was supported by the Vietnam National Foundation for Science and Technology Development (NAFOSTED) under grant number 101.02-2021.42. 
\end{acknowledgement}

\bibliographystyle{plain}

\end{document}